\documentclass{article}

\pdfoutput = 1

\usepackage[colorlinks=true]{hyperref}
\usepackage{mathrsfs}
\usepackage{amssymb,amsmath,amsfonts,amsthm}
\usepackage[all]{xy}
\xyoption{2cell}
\UseTwocells
\usepackage{multirow,array}

\newcommand{\cA}{\mathcal{A}}
\newcommand{\cE}{\mathcal{E}}
\newcommand{\cF}{\mathcal{F}}
\newcommand{\cG}{\mathcal{G}}
\newcommand{\cH}{\mathcal{H}}
\newcommand{\cI}{\mathcal{I}}
\newcommand{\cK}{\mathcal{K}}
\newcommand{\cL}{\mathcal{L}}
\newcommand{\cM}{\mathcal{M}}
\newcommand{\cN}{\mathcal{N}}
\newcommand{\cO}{\mathcal{O}}
\newcommand{\cT}{\mathcal{T}}
\newcommand{\coT}{\cT^\vee}
\newcommand{\cV}{\mathcal{V}}
\newcommand{\cExt}{\mathcal{E}xt}
\newcommand{\cEnd}{\mathcal{E}nd}

\newcommand{\dCE}{\delta_{\mathrm{CE}}}
\renewcommand{\d}{\delta}
\newcommand{\ddr}{\mathrm{d}}
\newcommand{\dtw}{\d_{\twist}}
\newcommand{\dpot}{\d_{\Pot}}
\newcommand{\id}{\mathrm{id}}

\newcommand{\cour}[1]{[\![ #1 ]\!]}

\newcommand{\Omcl}{\Omega_{\mathrm{cl}}}
\newcommand{\C}{\mathrm{C}} 
\newcommand{\Pot}{\mathrm{Pot}} 
\newcommand{\Z}{\mathrm{Z}} 
\newcommand{\coH}{\mathrm{H}} 
\newcommand{\rsect}{\Gamma}

\newcommand{\Shuf}{S}
\newcommand{\Cone}{\mathrm{Cone}}
\newcommand{\Hom}{\mathrm{Hom}}
\newcommand{\LA}{\mathsf{LA}}
\newcommand{\MC}{\mathrm{MC}}
\newcommand{\Perf}{\mathrm{Perf}}
\newcommand{\QCoh}{\mathrm{QCoh}}
\newcommand{\SSet}{\mathrm{SSet}}
\newcommand{\Sym}{\mathrm{Sym}}
\newcommand{\Tot}{\mathrm{Tot}}
\newcommand{\TCAlgd}{\mathsf{TCA}}
\newcommand{\TCAConn}{\mathsf{TCA_{conn}}}
\newcommand{\CAlgd}{\mathsf{CA}}
\newcommand{\ExDirac}{\mathsf{ExDir}}
\newcommand{\Dirac}{\mathsf{Dir}}
\newcommand{\TDirac}{\mathsf{TDir}}
\newcommand{\TDiracconn}{\mathsf{TDir_{conn}}}
\newcommand{\symp}{\mathsf{Symp}}
\newcommand{\Symp}{\mathsf{SA}}
\newcommand{\SympIso}{\mathsf{SA}^{\mathsf{iso}}}
\newcommand{\Lag}{\mathsf{Lag}}
\newcommand{\LagIso}{\mathsf{Lag}^{\mathsf{iso}}}

\DeclareMathOperator{\twist}{Tw}
\DeclareMathOperator*{\tens}{\otimes}
\newcommand{\utens}[1]{\underset{#1}{\otimes}}

\newcommand{\lie}[1]{\mathscr{L}_{#1}}

\DeclareMathOperator{\img}{\mathrm{img}}
\DeclareMathOperator{\ad}{\mathrm{ad}}
\newcommand{\CC}{\mathbb{C}}
\newcommand{\RR}{\mathbb{R}}
\newcommand{\KK}{\mathbb{K}}

\newcommand{\red}{\mathrm{red}}

\newcommand{\XL}{{[X/\cL]}}
\newcommand{\XT}{{[X/\cT_X]}}

\newcommand{\UL}{{[U/\cL]}}
\newcommand{\UM}{{[U/\cM]}}
\newcommand{\YM}{{[Y/\cM]}}
\newcommand{\YT}{{[Y/\cT_Y]}}
\newcommand{\VM}{{[V/\cM]}}

\newcommand{\g}{\mathfrak{g}}
\newcommand{\h}{\mathfrak{h}}

\newcommand{\tg}{\widetilde{g}}

\newcommand{\bomega}{\overline{\omega}}
\newcommand{\bbeta}{\overline{\beta}}

\newcommand{\rbrac}[1]{\left(#1\right)}
\newcommand{\abrac}[1]{\left\langle#1\right\rangle}
\newcommand{\set}[2]{\left\{#1\,\middle|\,#2\right\}}

\newcommand{\mapdef}[5]{
	\begin{array}{ccccc}
	#1 &:& #2 &\to& #3 \\
		&&  #4 &\mapsto& #5
	\end{array}
}

\newenvironment{psmallmatrix}
  {\left(\begin{smallmatrix}}
  {\end{smallmatrix}\right)}

\usepackage{color}
\definecolor{tocolor}{rgb}{.1,.1,.1}
\definecolor{urlcolor}{rgb}{.2,.2,.6}
\definecolor{linkcolor}{rgb}{.1,.1,.5}
\definecolor{citecolor}{rgb}{.4,.2,.1}
\definecolor{gray}{rgb}{.8,.8,.8}
\hypersetup{backref=true, colorlinks=true, urlcolor=urlcolor, linkcolor=linkcolor, citecolor=citecolor}

\usepackage{aliascnt}

\newcommand{\thdef}[2]{
	\newaliascnt{#1}{thm}  
	\newtheorem{#1}[#1]{#2}
	\aliascntresetthe{#1}  
	\newtheorem*{#1*}{#2}
	\expandafter\newcommand\expandafter{\csname #1autorefname\endcsname}{#2}
}

\newtheorem{thm}{Theorem}[section]
\thdef{lm}{Lemma}
\thdef{cor}{Corollary}
\thdef{conj}{Conjecture}
\thdef{prop}{Proposition}
\theoremstyle{definition}
\thdef{defn}{Definition}
\newtheorem*{theorem}{Theorem}

\theoremstyle{remark}
\thdef{remark}{Remark}
\thdef{ex}{Example}

\newcommand{\defterm}[1]{\textbf{\emph{#1}}}

\usepackage{titling}
\usepackage{abstract}
\begin{document}

\title{Shifted symplectic Lie algebroids\vspace{-0.2cm}}

\author{Brent Pym \and Pavel Safronov}
\date{}
\maketitle
\vspace{-0.7cm}
\begin{abstract}
Shifted symplectic Lie and $L_\infty$ algebroids model formal neighbourhoods of manifolds in shifted symplectic stacks, and serve as target spaces for  twisted variants of classical AKSZ topological field theory.  In this paper, we classify zero-, one- and two-shifted symplectic algebroids and their higher gauge symmetries,  in terms of classical geometric  ``higher structures'', such as Courant algebroids twisted by $\Omega^2$-gerbes.  As applications, we produce new examples of twisted Courant algebroids from codimension-two cycles, and we give symplectic interpretations for several well known features of higher structures (such as twists, Pontryagin classes, and tensor products).  The proofs are valid in the $C^\infty$, holomorphic and algebraic settings, and are based on a number of technical results on the homotopy theory of $L_\infty$ algebroids and their differential forms, which may be of independent interest.
\end{abstract}

\setcounter{tocdepth}{2}
{\footnotesize \tableofcontents}

\section{Introduction}

An $L_\infty$ algebroid $\cL$ on a manifold or variety $X$ is a model for the infinitesimal action of a higher groupoid on $X$.  The quotient space $\XL$ is a complicated object (a higher stack), but it has a convenient algebraic description that allows us to avoid working directly with stacks.  For example, the functions on $\XL$ are modelled by the $\cL$-invariant functions on $X$, and this ring has a natural enhancement to a commutative differential graded algebra (cdga) that corresponds to the full cohomology of the structure sheaf of $\XL$---the so-called Chevalley--Eilenberg algebra.  The sheaf of differential forms on $\XL$ has a similar description via an appropriate Weil algebra. This correspondence between $L_\infty$ algebroids and cdgas is often phrased in terms of the NQ manifold $\cL[1]$, but we shall stick to the stacky notation in this paper.

When $\XL$ carries an $n$-symplectic form (i.e.~the corresponding  NQ manifold carries a nondegenerate closed two-form of cohomological degree $n$), it may be used as  the  target space for a classical $(n+1)$-dimensional topological field theory via the AKSZ transgression procedure~\cite{Alexandrov1997}.   Beginning with the works of Roytenberg \cite{Roytenberg2002} and  \v{S}evera \cite{Severa2005}, it was understood that the symplectic forms in question have natural interpretations in terms of differential-geometric ``higher structures''.  They gave a complete classification of such structures for $n\le 2$ by relating them to symplectic structures on degree-shifted cotangent bundles; see \autoref{tab:NQ}. The case $n=3$ was also worked out recently~\cite{Liu2016a}.

It was soon realized that low-dimensional AKSZ theories admit various modifications and twists that do not come directly from symplectic NQ manifolds.  For example, Klim\v{c}\'{i}k and Strobl \cite{Klimcik2002} showed that one can add an extra term to the action of the Poisson sigma model using a closed three-form $H\in\Omega^3(X)$. In this way one arrives at the notion of a twisted Poisson manifold, where the Jacobi identity for the Poisson bracket fails by a term involving $H$. Similarly, Hansen and Strobl \cite{Hansen2010} showed that Courant algebroids and their sigma models can be modified by a closed four-form $K\in\Omega^4(X)$.  In a different direction, Kotov, Schaller and Strobl~\cite{Kotov2005} described a model in which the nondegeneracy condition on the symplectic structure is relaxed, replacing  twisted Poisson manifolds with Dirac manifolds.
 
In this paper, we give a symplectic interpretation for these modified AKSZ theories, based on Pantev, To\"{e}n, Vaqui\'{e} and Vezzosi's notion of shifted symplectic structures in derived algebraic geometry~\cite{Pantev2013}.  Thus we weaken the notions of closure and nondegeneracy for a two-form on $\XL$ in a homotopy-coherent way, so that they become invariant under arbitrary quasi-isomorphisms of $L_\infty$ algebroids (i.e.~independent of the presentation of the quotient stack).  This yields all of the modified target spaces above, and a bit more.

Our main results concern the classification of such shifted symplectic $L_\infty$ algebroids, their gauge symmetries, and their associated Lagrangian subobjects, for shifts up to two.  Since $L_\infty$ algebroids are equivalent to formal moduli problems~\cite{Nuiten2017}, our results could be viewed as normal forms for arbitrary shifted symplectic stacks in the formal neighbourhood of a smooth manifold, analogous to Weinstein's neighbourhood theorems in classical symplectic geometry~\cite{Weinstein1971,Weinstein1981}.

\begin{table}[t]
\caption{Classification of NQ manifolds with degree-$n$ symplectic forms}\label{tab:NQ}
\begin{center}\def\arraystretch{1.2}
\begin{tabular}{cccc}
$n$ & NQ manifold & Differential geometric data & AKSZ field theory \\
\hline
0 & $X$ & $X$ is a symplectic manifold & Classical mechanics\\ 
1 & $\coT_X[1]$ & $X$ is a Poisson manifold & Poisson sigma model\\ 
2 & $M\subset \coT_X[2]\cE[1]$ & $\cE$ is a Courant algebroid~\cite{Liu1997} & Courant sigma model \\
\end{tabular}
\end{center}
\end{table}

To describe the results in more detail, let us begin by recalling that as a higher stack, the quotient $[X/\cL]$ has a tangent complex (rather than a tangent bundle), modelled by the derivations of the Chevalley--Eilenberg algebra. When pulled back along the projection $\pi\colon X \to \XL$, it gives the complex
\[
\pi^*\cT_\XL \cong \rbrac{\xymatrix{
\cdots \ar[r] & \cL_2 \ar[r] & \cL_1 \ar[r] & \cL_0 \ar[r] & \cT_X
}}
\]
where the tangent bundle $\cT_X$ of $X$ sits in degree zero, and the remaining components give the complex of vector bundles underlying the $L_\infty$ algebroid $\cL$.  In particular, the cohomology in degree zero gives the normal spaces to the $\cL$-orbits, or equivalently the Zariski tangent spaces to the quotient.

Adapting the definitions in \cite{Pantev2013} to the present setting, we may weaken the notion of nondegeneracy of a two-form, so that a degree-$n$ symplectic structure corresponds to an $n$-cocycle in the complex $\Omega^{2}_\XL$ that induces a quasi-isomorphism between  $\cT_\XL$ and the shifted cotangent complex $\coT_\XL[n]$, rather than a strict isomorphism.  We may also weaken the closure condition on forms, allowing a $p$-form $\omega_p$ to be closed only up to coherent homotopy
\begin{align}
\delta\omega_p &= 0 \nonumber \\
\ddr \omega_p &= -\delta \omega_{p+1} \label{eqn:htpy-close} \\
\ddr \omega_{p+1} &= -\delta \omega_{p+2}.\nonumber \\
&\ \vdots \nonumber
\end{align}
Here $\ddr$ is the de Rham differential, and $\delta$ is the Chevalley--Eilenberg differential that plays the role of the \v{C}ech differential on $\Omega^\bullet_\XL$. 

In this way, we obtain the notion of an $n$-shifted symplectic structure: a weakly nondegenerate two-form of degree $n$ that is closed up to homotopy.  The notions of isotropic and Lagrangian subobjects, symplectomorphisms, etc.,~are similarly weakened.  In particular, when all of the higher homotopies and symmetries are accounted for, shifted symplectic algebroids naturally form the objects of a whole $\infty$-groupoid, rather than a discrete set.

Thus the cost of having an invariant notion of symplectic structure on $\XL$ is the proliferation of closure and symmetry data, satisfying equations that are typically underdetermined due to the weak notion of nondegeneracy.  One therefore desires stricter models that will compress this complicated information into a small amount of classical differential-geometric data.  In \autoref{sec:defs} and \autoref{sec:quotient-forms}, we establish several basic technical results about $L_\infty$ algebroids and their differential forms that facilitate such strictifications:
\begin{itemize}
\item \autoref{thm:mapping-space} gives an efficient model for the  $\infty$-groupoid of $L_\infty$ algebroid morphisms and their higher homotopies, via Getzler's results~\cite{Getzler2009} on gauge fixing for Maurer--Cartan sets.
\item \autoref{thm:homotopytransfer} extends the Homotopy Transfer Theorem, to show that $L_\infty$ algebroid structures may be transferred along quasi-isomorphisms of complexes of vector bundles.  This eventually allows us to simplify shifted symplectic algebroids by replacing them with quasi-isomorphic models.
\item \autoref{thm:reduced} shows that the complex $\Omega^{\ge p,\bullet}(\XL)$ of homotopy closed $p$-forms retracts onto a much smaller complex that is amenable to computation.  It implies that  the whole sequence~\eqref{eqn:htpy-close} of closure data for a $(p,q)$-form  $\omega_p$ reduces to a single cocycle for $\coH^q(X,\Omega^{\ge p}_X)$, and explains why  ``twists'' of higher structures always occur through expressions such as $\iota_x \iota_y H$, in which vectors are contracted into forms.
\end{itemize}

With these results in hand, we proceed to give strict models for the full $\infty$-groupoids of shifted symplectic $L_\infty$ algebroids, in terms of classical higher structures.  Specializing our results to the level of isomorphism classes, we obtain the following statement.
\begin{theorem}
Let $X$ be a $C^\infty$ manifold, a complex manifold or a smooth algebraic variety.  Then $n$-shifted symplectic $L_\infty$ algebroids on $X$ for $n \le 2$ are determined up to symplectic quasi-isomorphism by the geometric data shown in \autoref{tab:shifted}.
\end{theorem}
Notice that the symplectic NQ manifolds of \autoref{tab:NQ} appear in different columns of \autoref{tab:shifted}, depending on the degree.

\begin{table}[h]\newcolumntype{M}[1]{>{\centering\arraybackslash}m{#1}} 
\caption{Classification of shifted symplectic $L_\infty$ algebroids}\label{tab:shifted}
\begin{center}\def\arraystretch{1.2}
\begin{tabular}{M{1cm}|M{3.5cm}M{3.5cm}M{2cm}}
\multirow{2}{*}{Shift} & \multicolumn{3}{c}{Structure of the quotient map $X \to \XL$} \\
 &  None & Isotropic & Lagrangian \\ 
\hline
0 & Regular foliation $\cL \subset \cT_X$ with symplectic leaf space  &  $\cL = \cT_X$ &  $\dim X = 0$ \\ \hline
1 & Dirac structure $\cL$ in an exact Courant algebroid $\cE$ & $\cE \cong \cT_X \oplus \cT_X^\vee$ & Poisson structure \\ \hline
2 & $\cL \cong (\coT_X \to \cE)$ for a twisted Courant algebroid $\cE$ & $\cE$ is an (untwisted) Courant algebroid  & $\cE = 0$
\end{tabular}
\end{center}
\end{table}

The proof of the classification illustrates the intriguing recursive nature of shifted symplectic structures: given sufficient knowledge of $n$-shifted symplectic structures and their Lagrangians, one can immediately deduce nontrivial information about $(n-1)$-shifted symplectic structures.  Thus we treat the two-shifted case first, and work our way down.

 We prove that an arbitrary two-shifted symplectic $L_\infty$ algebroid always has a quasi-isomorphic model that is a two-term complex $\cL \cong (\coT_X \to \cE)$ equipped with a strictly nondegenerate two-form, giving a self-dual isomorphism
\begin{align}
\begin{gathered}
\xymatrix{
\pi^*\cT_\XL \ar[d] & \coT_X \ar[r] \ar@{=}[d] & \cE \ar[r] \ar[d] & \cT_X \ar@{=}[d] \\
\pi^*\coT_\XL[2] & \coT_X \ar[r] & \cE^\vee \ar[r] & \cT_X.
}
\end{gathered}
\end{align}
This stands in contrast with other shifts, where such strict nondegeneracy is, in general, impossible.  

 At this point, we are nearly in the setting of~\cite{Roytenberg2002,Roytenberg1998,Severa2005}, where the remaining data of a symplectic NQ manifold is encoded in a binary Courant--Dorfman bracket on $\cE$, making it into a Courant algebroid.  But there are two key differences.  Firstly, the nontrivial closure data means that the Jacobi identity for the Courant--Dorfman bracket may fail by a term involving a four-form as in \cite{Hansen2010}.  Secondly, the possibility of gluing by higher gauge transformations means that the transition functions for the bundle $\cE$ and its bracket may not satisfy the usual cocycle condition, but rather a twisted cocycle condition involving an $\Omega^2$-gerbe.  We therefore arrive at the following result:

\begin{theorem}[see \autoref{sec:courant} and \autoref{thm:shift2}]
The $\infty$-groupoid of two-shifted symplectic $L_\infty$ algebroids  on $X$ is equivalent to the following strict two-groupoid:
\begin{itemize}
\item[] \textbf{Objects} are Courant algebroids $\cE$ twisted by classes in $\coH^2(X,\Omega^{\ge 2}_X)$.

\item[] \textbf{1-morphisms} are given by maps $f\colon \cE_1\rightarrow \cE_2$ that preserve the pairing and anchors identically, and preserve the brackets and gluing maps up to a coboundary $H \in \C^1(X,\Omega^{\ge 2}_X)$.

\item[] \textbf{2-morphisms} are elements $B \in \C^0(X,\Omega^2_X)$ that shear the bundle maps
\end{itemize}
\end{theorem}
\autoref{thm:shift2lagrangian} gives an analogous description of the $\infty$-groupoid of two-shifted Lagrangians $\YM \to \XL$ where $Y \subset X$ is a closed submanifold.  Now the objects are pairs $(\cE,\cF)$ of a twisted Courant algebroid $\cE$ on $X$ and a maximal isotropic twisted subbundle $\cF \subset \cE|_Y$ that is involutive for the bracket.  This extends the notion of a Dirac structure with support~\cite{Alekseev,Bursztyn2009,Severa2005}, allowing for twists by  relative cohomology classes of the pair $(X,Y)$.

The rest of the paper consists of applications of the two-shifted classification.  The first application, in \autoref{sec:shift2-ex}, is the construction of  several new examples of twisted Courant algebroids.  In particular, in the complex analytic or algebraic settings, any smooth codimension-two cycle $Y \subset X$ gives rise to a twisted Courant algebroid whose twisting class is the cycle class $[Y] \in \coH^{2,2}(X)$, by a procedure reminiscent of Serre's construction of rank-two vector bundles.  We expect that this construction will allow fivebranes (curves in Calabi--Yau threefolds) to be incorporated into the recent works~\cite{Anderson2014,Baraglia2015,delaOssa2016,Garcia-Fernandez2015} relating  holomorphic Courant algebroids to heterotic string theory.

The second application is the classification of one-shifted symplectic structures in terms of Dirac structures in exact Courant algebroids
(\autoref{thm:shift1}).  This approach illuminates the dual nature of such Dirac structures: on the one hand, they correspond to  Lagrangians in certain degree-two symplectic NQ manifolds~\cite{Severa1998--2000}; on the other hand, they are the linearizations of quasi-symplectic Lie groupoids~\cite{Bursztyn2004,Xu2004}, and by definition, the latter are exactly the atlases for one-shifted symplectic stacks.

The remaining applications illustrate how various well known features of Dirac structures and Courant algebroids can be viewed through the shifted symplectic lens.   The classification of isotropic quotients immediately recovers  \v{S}evera's cohomological classification~\cite{Severa1998--2000} of exact Courant algebroids (\autoref{cor:exact-courant}), and the link~\cite{Bressler2007,Severa1998--2000} between Courant extensions and the first Pontryagin class (\autoref{cor:pontryagin}).  The Courant trivializations and ``generalized tangent bundles'' defined in the context of  generalized complex branes~\cite{Gualtieri2011} are interpreted as Lagrangians (\autoref{cor:gen-tan} and \autoref{sec:shift1-lagrangian}).  Finally, the formalism of derived Lagrangian intersections gives a perspective on both the tensor product~\cite{Alekseev2009,Gualtieri2011}  of Dirac structures (\autoref{sec:shift1-symp}), and the equivalence~\cite{Crainic2004,Weinstein1987} between symplectic groupoids and Poisson manifolds (\autoref{sec:Poisson}).

\paragraph*{Acknowledgements:} We would like to thank X.~de la Ossa, M.~Garcia-Fernandez, M.~Gualtieri, N.~Hitchin and R.~Mehta for useful conversations. B.P.~was supported by EPSRC Grant EP/K033654/1. P.S.~was partially supported by EPSRC Grant EP/I033343/1.

\section{Basics of $L_\infty$ algebroids}
\label{sec:defs}

Throughout the paper, a \defterm{manifold} is a real $C^\infty$ manifold, a complex manifold, or a smooth algebraic variety over a field of characteristic zero.  We denote by $\KK$ be the base field, which is $\RR$, $\CC$, or an arbitrary field $\KK$ of characteristic zero, respectively.  The structure sheaf $\cO_X$ is the corresponding sheaf of $C^\infty$, holomorphic, or algebraic functions.

We also consider the subclass of \defterm{affine manifolds}, which are arbitrary $C^\infty$-manifolds, Stein manifolds, and smooth affine varieties, respectively.  The essential features of affine manifolds that we need are that extensions of vector bundles always split, and appropriate sheaves of modules are acyclic for sheaf cohomology.  Moreover, every manifold has a good cover by affine manifolds. 

In this section, we cover the basic definitions and properties of $L_\infty$ algebroids on arbitrary manifolds and the corresponding quotient stacks.  We focus first on the affine case, and then extend to general manifolds by gluing along affine covers.

\subsection{Affine $L_\infty$ algebroids}

\label{sect:Linfalgd}

Let $U$ be an affine manifold, and consider a bounded complex
\[
\cL = \rbrac{\xymatrix{
\cdots \ar[r]^-{\d} & \cL_2 \ar[r]^-{\d} & \cL_1 \ar[r]^-{\d} & \cL_0
}}
\]
of finite rank vector bundles, where $\cL_i$ sits in degree $-i$.  Here and throughout, we make no notational distinction between a vector bundle and its corresponding  finite rank projective $\cO(U)$-module of sections.

\begin{defn}
\label{def:Liealgebroid}
An \defterm{$L_\infty$ algebroid structure on $\cL$} is an $L_\infty$ algebra structure on the $\KK$-vector space $\cL$, determined by a $\KK$-bilinear bracket
\[
[-,-] : \cL \times \cL \to \cL
\]
of degree zero and a collection of $\cO(U)$-multilinear brackets
\[
[-,\ldots,-] \colon \underbrace{\cL \times \cdots \times \cL}_{n\textrm{ times}} \to \cL, \qquad n > 2
\]
of degree $2-n$, together with an $\cO(U)$-linear morphism of complexes
\[
a\colon \cL \to \cT_U,
\]
subject to the Leibniz rule
\[
[x,fy] = (\lie{ax}f)y + f[x,y]
\]
for $x,y \in \cL$ and $f \in \cO(U)$.  An $L_\infty$ algebroid is a \defterm{Lie $n$-algebroid on $U$} if the underlying complex is concentrated in degrees $[-(n-1),\ldots,0]$.
\end{defn}

In what follows, we will often suppress the anchor and brackets from notation and simply say that $\cL$ is an $L_\infty$ algebroid on $U$.  We remark that this notion of an $L_\infty$ algebroid is often referred to as a \emph{split} $L_\infty$ algebroid; see, e.g.,~\cite{Sheng2011}.

We recall that by an $L_\infty$ structure, we mean that the brackets are graded skew-symmetric, and that the higher Jacobi identity
\begin{align*}
\sum_{i+j = n+1} \sum_{\sigma \in \Shuf(i,j-1)}(-1)^{\sigma + \epsilon}(-1)^{i(j-1)} [[x_{\sigma_1},\ldots,x_{\sigma_i}],x_{\sigma_{i+1}},\ldots,x_{\sigma_n}] =0
\end{align*}
holds for all $n$.  Here $[-]_1 = \d$ is the differential on the underlying complex,  $\Shuf(i,j-1)$ denotes the set of $(i,j-1)$-unshuffles, $(-1)^{\sigma}$ is the sign of the permutation $\sigma$, and  $(-1)^{\epsilon}$ is determined by the Koszul sign rule.

For any $L_\infty$ algebroid, the anchor map is automatically compatible with the brackets, in the sense that $a[x,y] = [ax,ay]$.  Thus it gives rise to an action of the $L_\infty$ algebra $\cL$ on $\cO(U)$ by derivations, i.e.~an infinitesimal action of $\cL$ on $U$.  We denote the quotient by
\[
\pi\colon U \to \UL.
\]
This quotient can be defined rigorously as a formal higher stack \cite{Nuiten2017}. For example, when $\cL$ is concentrated in degree zero, it is a Lie algebroid in the ordinary sense, and $\UL$ is the quotient of $U$ by the formal groupoid integrating $\cL$.

Functions on the quotient are given by the $\cL$-invariant functions
\[
\coH^0(\UL,\cO_\UL) = \cO(U)^\cL = \set{f \in \cO(U)}{\lie{ax}f = 0 \textrm{ for all }x \in \cL}
\]
However, the quotient is not affine, and to fully capture its geometry (e.g.~the higher cohomology of the structure sheaf),  we need to replace the ring of invariants with a derived enhancement thereof. 

 As is well known, this enhancement is provided by the commutative differential graded algebra (cdga) known as the \defterm{Chevalley--Eilenberg algebra}
  \[
 \cO(\UL) := (\Sym_{\cO(U)}(\cL^\vee[-1]),\delta).
 \] 
 Here $\cL^\vee[-1]$ is the dual complex, with a shift, and $\Sym$ denotes the graded symmetric algebra.   When $\cL$ is concentrated in degree zero, we have that $\cO(\UL) = \wedge^\bullet \cL^\vee$ with $\cL^\vee$ in degree one.  But in general, $\cO(\UL)$ has two distinct gradings: the first is the internal degree induced by the grading on $\cL$, and the second is the \defterm{weight}, defined so that the $n$th symmetric power has weight $n$.   Thus an element of $\cL_j^\vee$ defines a generator for $\cO(\UL)$ of degree $j+1$ and weight one.  We denote the weight-$n$ component of an element $u \in \cO(\UL)$ by $u_n$.

The differential $\delta$ on  $\cO(\UL)$ has degree one, but components of many weights, corresponding to the differential on $\cL$ and the higher Lie brackets.  Viewing elements $u \in \cO(\UL)$ as multilinear operators on $\cL$, the differential $\d u$ is defined by its weight components
\begin{align*}
(\d u)_n(x_1,\ldots,x_n) &= (\dCE u)_n(x_1,\ldots,x_n)
+ \sum_{i=1}^n (-1)^\epsilon \lie{a x_i } u_{n-1}(x_1,\ldots,\widehat{x_i},\ldots,x_n)
\end{align*}
where $\dCE u$ is the $\KK$-multilinear operator $\cL\times \cdots \times \cL \to \cO(U)$ defined by summing over all possible insertions of brackets into $u$:
\begin{align}
(\dCE u)_n(x_1,\ldots,x_n) = \sum_{i+j=n+1} \sum_{\sigma \in \Shuf(i,j-1)} (-1)^{{i \choose 2}+\epsilon} u_j([x_{\sigma_1},\ldots,x_{\sigma_i}]_i,x_{\sigma_{i+1}},\ldots,x_{\sigma_n}) .\label{eqn:CEterm}
\end{align}
Thus $\delta$ is simply the Chevalley--Eilenberg differential of $\cO(U)$ when viewed as a module over the $L_\infty$ algebra $\cL$, but restricted to the $\cO(U)$-multilinear cochains. By \cite{Bonavolonta2013}, this formula gives a bijective correspondence between $L_\infty$ algebroid structures on  $\cL$ and  differentials on $\Sym(\cL^\vee[-1])$ that make it into a cdga (often called $Q$-structures or homological vector fields).

There is a natural morphism of cdgas
\[
\pi^*\colon \cO(\UL) \to \cO(U)
\]
given by projection on the degree-zero component.  It corresponds geometrically to the pullback of functions along the quotient map $\pi : X \to \XL$.  Indeed, taking the zeroth cohomology, we obtain the inclusion
\[
\coH^0(\pi^*)\colon \cO(U)^\cL \to \cO(U) 
\]
of invariant functions.  More generally, the Chevalley--Eilenberg differential $\delta$ plays the role of the \v{C}ech differential on $\UL$, so that
\[
\coH^\bullet(\cO(\UL)) = \coH^\bullet(\UL,\cO_\UL)
\]
is the full cohomology of the structure sheaf of the quotient stack.  When $\cL$ is a classical Lie algebroid, this follows from a presentation of the formal groupoid integrating a given Lie algebroid in terms of its jet algebra and \cite[Theorem 3.20]{Kowalzig2011} which identifies groupoid cohomology and Lie algebroid cohomology. We expect this result to hold for general $L_\infty$ algebroids as well.

\subsection{Morphisms of affine $L_\infty$ algebroids}

Suppose now that $\cL$ and $\cM$ are $L_\infty$ algebroids on an affine manifold $U$. 

\begin{defn} A \defterm{(base-preserving) $L_\infty$-morphism $f : \cM\to \cL$} is a morphism of cdgas
\[
f^*\colon \cO(\UL) \to \cO(\UM)
\]
that acts as the identity on the degree-zero part, i.e.~on $\cO(U)$.
\end{defn}

So a base-preserving morphism corresponds to a commutative diagram 
\[
\xymatrix{
& U \ar[ld]_{\pi_\cM} \ar[rd]^{\pi_\cL}& \\
\UM \ar[rr]^-{f} && \UL
}
\]
of the quotients.  By \cite{Bonavolonta2013}, such a morphism may be described explicitly by a sequence  of $\cO(U)$-multilinear graded-skew-symmetric  maps
\[
f_n\colon \underbrace{\cM \times \cdots \times \cM}_{n\textrm{ times}}
 \to \cL, \qquad n \ge 1
\]
of degree $1-n$, with the following properties:
\begin{enumerate}
\item $f_1\colon \cM \to \cL$ is a morphism of complexes of vector bundles that is compatible with the anchors, i.e.~$a_\cM f_1 = a_\cL$.
\item The sequence $f_1,f_2,\ldots,$ defines an $L_\infty$ morphism of the $\KK$-linear $L_\infty$ algebras underlying $\cL$ and $\cM$, i.e.~we have the usual bracket compatibilities
\begin{align*}
& \sum_{i+j=n+1}\sum_{\sigma \in \Shuf(i,j-1)} (-1)^{\sigma+(i-1)j} f_j([x_{\sigma_1},\ldots,x_{\sigma_i}],x_{\sigma_{j+1}},\ldots,x_{\sigma_n}) \\
=& \sum_{\stackrel{j \ge 1}{i_1+\cdots+i_j = k}} \sum_{\sigma \in \Shuf(i_1,\ldots,i_k)} (-1)^{\sigma+\tau}[f_{i_1}(x_1,\ldots,x_{i_1}),\ldots,f_{i_j}(x_{\sigma_{n-i_j+1}},\ldots,x_{\sigma_n})]
\end{align*}
where $\tau = (j-1)(i_1-1)+\cdots+1\cdot(i_{j-1}-1)$.
\end{enumerate}
An $L_\infty$ morphism is an \defterm{$L_\infty$ quasi-isomorphism} if $f_1$ is a quasi-isomorphism of complexes of vector bundles.

There are also notions of homotopies between morphisms, homotopies between homotopies, etc.~yielding a whole simplicial set $\Hom(\cM,\cL)_\bullet$ that models the mapping space.  It is described by the following standard construction.   (See \cite{Dolgushev2015} for the simplicial category of $L_\infty$ algebras.)  Let $\Omega^\bullet(\Delta^n)$ be the algebra of polynomial functions on the $n$-simplex, i.e.
\[
\Omega^\bullet(\Delta^n) = \KK[t_0, ..., t_n, \ddr t_0,\ldots,\ddr t_n] / \rbrac{\sum t_i = 1 , \sum \ddr t_i = 0}
\]
The collection $\Omega_\bullet = \{\Omega^\bullet(\Delta^n)\}_n$ forms a simplicial cdga with the usual face and degeneracy maps.  We then define the set of $n$-simplices $\Hom(\cM,\cL)_\bullet$ to be the set of $\Omega_\bullet$-linear cdga maps $ \Omega_\bullet \otimes \cO(\UL) \to \Omega_\bullet \otimes \cO(\UM)$ that act trivially on the summand $\Omega_\bullet \otimes \cO(U)$.

By construction an $n$-simplex of $\Hom(\cM,\cL)_\bullet$ is an $L_\infty$-morphism of $\Omega_n$-linear $L_\infty$ algebroids $\cM\otimes \Omega_n \to\cL\otimes \Omega_n$. Due to $\Omega_n$-linearity this is the same as a morphism of $\KK$-linear $L_\infty$ algebras $\cM \to \cL \otimes_{\KK} \Omega_n$ such that the constituent maps $f_i \colon \wedge^i \cM \to \cL \otimes \Omega_n$ are $\cO(U)$-multilinear, and  the component $f_1\colon \cM \to \cL \otimes \Omega_n$ is compatible with the anchors. In other words, $f_1$ projects to $a_\cM \otimes 1$ under the natural map
\[
\Hom_\KK(\cM,\cL\otimes\Omega_n) = \Hom_\KK(\cM,\cL)\otimes \Omega_n \to \Hom(\cM,\cT_U)\otimes \Omega_n
\]
given by composing with the anchor $a_\cL\colon \cL \to \cT_U$.

With this simplicial set of morphisms, it is evident that $L_\infty$ algebroids on $U$ from a simplicial category (a category enriched in simplicial sets).  However, these simplicial sets are quite large, and it is useful to have a more efficient model for the homotopy type of the mapping space.  To this end, we follow Getzler and use Dupont's gauge operator
\[
s_n\colon \Omega^k(\Delta^n)\rightarrow \Omega^{k-1}(\Delta^n);
\]
to cut down on the redundancy; we refer to \cite[Section 3]{Getzler2009} for the detailed description of this operator.  Restricting to morphisms that satisfy the gauge condition
\[
(s_\bullet\otimes \id)f^*(1\otimes g)  = 0
\]
for any $g\in \cO(\UL)$, we obtain a simplicial subset
\[
\Hom_\red(\cM,\cL)\subset \Hom(\cM,\cL)
\]
whose relevance is illustrated by the following result.

\begin{thm}\label{thm:mapping-space}
Let $U$ be an affine manifold.  For any $L_\infty$ algebroids $\cL$ and $\cM$ on $U$, the following statements hold:
\begin{enumerate}
\item The simplicial set $\Hom(\cM, \cL)$ is a Kan complex, i.e.~a weak $\infty$-groupoid.

\item The inclusion $\Hom_{\red}(\cM, \cL)\subset \Hom(\cM,\cL)$ is a weak equivalence.

\item If $\cM$ is a Lie $n$-algebroid, then $\Hom_\red(\cM,\cL)$ is an $(n-1)$-groupoid.
\end{enumerate}
\label{thm:Linfhomotopytype}
\end{thm}
\begin{proof}
If $\Hom(\cM,\cL)_\bullet$ is empty, the statements are vacuous.  Thus we may assume that it is nonempty and fix an $L_\infty$ algebroid morphism $f\colon \cM\rightarrow \cL$. Using $f$, we will express $\Hom(\cM,\cL)_\bullet$ as the Maurer--Cartan simplicial set $\MC_\bullet(\h)$ of a nilpotent $L_\infty$ algebra $\h$ concentrated in degrees $(-(n-1),\infty)$, so that the three statements follow directly from Proposition 4.7, Corollary 5.11 and Theorem 5.4 in \cite{Getzler2009}, respectively.

To begin, consider the simplicial set $\Hom_\KK(\cM,\cL)_\bullet$ of all $\KK$-linear $L_\infty$ morphisms $\cM \to \cL \otimes \Omega_\bullet$.  We recall from, e.g.~\cite[Section 3]{Dolgushev2014} or \cite[Section 3]{Shoikhet2007}, that $\Hom_\KK(\cM,\cL)_\bullet$ is the Maurer--Cartan simplicial set of a $\KK$-linear $L_\infty$ algebra
\[
\g =  \Hom_\KK(\Sym^{\geq 1}_\KK(\cM[1]), \cL).
\]
The action of the differential $\delta_\g$ on $\psi \in \g$ is given by its weight components
\begin{align}
(\delta_\g \psi)(x_1,\ldots,x_m) &= \delta_\cL(\psi_m(x_1,\ldots,x_m)) + (\dCE \psi)(x_1,\ldots,x_n)\label{eqn:map-diff}
\end{align}
where $\dCE$ is defined by the same formula as in \eqref{eqn:CEterm}.  Meanwhile, the binary bracket of $\psi,\phi \in \g$ is given by
\begin{align}
[\psi,\phi]_{\g}(x_1,\ldots,x_m) &= \sum_{i+j=m}\sum_{\sigma \in \Shuf(i,j)} (-1)^\epsilon [\psi_i(x_{\sigma_1},\ldots,x_{\sigma_i}),\phi_j(x_{\sigma_{i+1}},\ldots,x_{\sigma_m})]_\cL. \label{eqn:map-brac}
\end{align}
where $x_1,\ldots,x_m \in \cM$, and the higher brackets are defined similarly, by composition with the higher brackets on $\cL$.

Consider the vector subspace
\[
\h = \Hom_{\cO(U)}(\Sym_{\cO(U)}^{\ge 1}(\cM[1]),\cK) \subset \g
\]
consisting of $\cO(U)$-multilinear maps that take values in $\cK = \ker a_\cL \subset \cL$.  Evidently elements of $\Hom(\cM,\cL)_\bullet$ can be identified with Maurer--Cartan elements in $\g\otimes\Omega_\bullet$ of the form $f \otimes 1 + g$ where $g \in \h \otimes \Omega_\bullet$.  Since $f$ itself is a Maurer--Cartan element, $g$ must be a Maurer--Cartan element for the $f$-twisted $L_\infty$ brackets
\begin{align}
[\psi_1,\ldots,\psi_k]_{\g,f} = \sum_{l \ge 0}\frac{1}{l!}[\underbrace{f,\ldots,f}_{l\textrm{ times}},\psi_1,\ldots,\psi_k]_\g. \label{eqn:MC-twist}
\end{align}
Thus to identify $\Hom(\cM,\cL)_\bullet = \MC_\bullet(\h)$, it is sufficient to show that $\h \subset \g$ is preserved by the twisted brackets.

To this end, suppose that $\psi_1,\ldots,\psi_k \in \h$ and consider the $\KK$-linear map
\[
[\psi_1,\ldots,\psi_k]_{\g,f} \colon \Sym^{\ge 1}_\KK\cM[1] \to \cL.
\]
Applying the anchor of $\cL$ to \eqref{eqn:map-diff} and \eqref{eqn:map-brac}, we see that this linear map automatically takes values in $\cK \subset \cL$, so it remains to check that it is  $\cO(U)$-multilinear.   For this, we observe that if $k \ge 2$, then every term on the right hand side of \eqref{eqn:MC-twist} is an $\cO(U)$-multilinear operator; indeed, such a term is either a binary bracket $[\psi_1,\psi_2]$, which is $\cO(U)$-bilinear  because $\psi_1,\psi_2$ are valued in $\cK$, or it is a higher bracket, and therefore automatically multilinear.  Similarly,  for $k=1$, we have the twisted differential
\begin{align*}
(\delta_{\g,f}\psi)(x_1,\ldots,x_m) &= -\sum_{i<j} (-1)^{\epsilon}\psi_m([x_i,x_j]_\cM,x_1,\ldots,\widehat x_i,\ldots,\widehat x_j,\ldots,x_m) \\
&\ \ \ + \sum_i (-1)^\epsilon [f(x_i),\psi_m(x_1,\ldots,\widehat x_i,\ldots,x_m)]_\cL + \cdots
\end{align*}
where $\cdots$ denotes terms that are manifestly $\cO(U)$-multilinear.  Using the Leibniz rule for the binary brackets on $\cL$ and $\cM$ and the fact that $f$ intertwines the anchors, we see that the nonlinearities cancel, as desired.
\end{proof}

This construction can be easily generalized to maps that do not preserve the base.  Suppose that $f\colon V \to U$ is a morphism of affine manifolds, and $\cL$ and $\cM$  are $L_\infty$ algebroids on $U$ and $V$.  A morphism $\cM \to \cL$ covering $f$ is an extension of $f^*\colon \cO(U) \to \cO(V)$ to a morphism of cdgas $\cO(\UL) \to \cO(\VM)$, with higher homotopies defined by tensoring with forms on simplices as above. We refer to \cite[Definition 4.1.6]{Bonavolonta2013} for a more explicit description of morphisms in terms of brackets.  The conclusion is  that every morphism $\cM \to \cL$ covering $f$ factors uniquely through the pullback $L_\infty$ algebroid $f^!\cL$, defined on the level of complexes by the fibre product
\[
f^!\cL = \cL \underset{f^*\cT_U}{\times} \cT_V.
\]
Hence this more general situation reduces to the base-preserving one.
\begin{remark}
When the anchor of $\cL$ is not transverse to $f$, one should be careful about the interpretation of this fibre product.  Either one views it as a complex of coherent sheaves, or better, one takes the full derived fibre product, which will have cohomology in degree one.  Such objects are slightly more general than the $L_\infty$ algebroids we have been considering, but the argument in \autoref{thm:mapping-space} is insensitive to this difference.
\qed
\end{remark}

\subsection{Homotopy transfer}

Recall the Homotopy Transfer Theorem: if two complexes of vector spaces are quasi-isomorphic, and one of them is an $L_\infty$ algebra, then so is the other, and the quasi-isomorphism extends to an equivalence of the $L_\infty$-structures. See, for example, \cite[Theorem 10.3.9]{Loday2012}.  In this section extend this result to algebroids:
\begin{thm}
\label{thm:homotopytransfer}
Let $\cL'$ and $\cL''$ be bounded complexes of vector bundles on an affine manifold $U$, and let $f\colon \cL' \to \cL''$ be a quasi-isomorphism.  Given an $L_\infty$ algebroid structure on $\cL''$, there exists an $L_\infty$ algebroid structure on $\cL'$ and an extension of $f$ to an $L_\infty$ quasi-isomorphism.
\end{thm}

Recall that a \defterm{special deformation retract} between bounded complexes $\cL$ and $\cM$ is a pair of morphisms
\[
\xymatrix{
 \cL \ar@/^/[r]^-{i} &  \cM \ar@/^/[l]^{p} 
}
\]
and a homotopy operator $h\colon \cM\rightarrow \cM[-1]$ satisfying the identities
\[
ip - \id = \d h + h\d, \qquad pi = \id, \qquad hi=0,\qquad ph=0,\qquad h^2=0.
\]
An arbitrary quasi-isomorphism of complexes may be factored into a pair of special deformation retracts using the mapping cylinder construction.  Thus, it is enough to prove \autoref{thm:homotopytransfer} in the case where the map $f \colon \cL' \to \cL''$ is one of the maps $i \colon \cL \to \cM$ or $p \colon \cM \to \cL$ in a special deformation retract.    These two cases are the content of the following two lemmas.

\begin{lm}
\label{lm:homotopytransfer1}
If $\cL$ is an $L_\infty$ algebroid, then there exists an $L_\infty$ algebroid structure on $\cM$ such that both $i \colon \cL\rightarrow \cM$ and $p\colon \cM\rightarrow \cL$ are $L_\infty$ quasi-isomorphisms.
\end{lm}
\begin{proof}
Use the deformation retract to split
\[
\cM\cong \cL\oplus \cK
\]
where $\cK  = \ker p$ is a contractible complex. Therefore, we can identify
\[\cK\cong \Cone(\cK'\stackrel{\id}\rightarrow \cK')\]
for a complex of vector bundles $\cK'$.

Pick a (not necessarily flat) $\cL$-connection on $\cK'$. Then by \cite[Example 3.8]{Abad2012} we obtain the data of a representation up to homotopy on $\cK$, i.e. a differential on $\Sym(\cL^\vee[-1])\otimes \cK$ making it into a dg module over the dg algebra $\cO(\UL)$.  Therefore we may extend the differential on $\Sym(\cL^\vee[-1])\otimes \cK$ to the symmetric algebra
\[
\Sym(\cM^\vee[-1])\cong \Sym(\cL^\vee[-1]) \utens{\cO(U)} \Sym(\cK^\vee[-1])
\]
resulting in an $L_\infty$ algebroid structure on $\cM$.  By construction, the projection
\[
\Sym(i^*)\colon \Sym(\cM^\vee[-1])\rightarrow \Sym(\cL^\vee[-1])
\]
is compatible with the differentials and similarly for $\Sym(p^*)$.  Hence $p$ and $i$ define strict quasi-isomorphisms of $L_\infty$ algebroids.
\end{proof}

\begin{lm}
\label{lm:homotopytransfer2}
If $\cM$ is an $L_\infty$ algebroid, then there exists an $L_\infty$ algebroid structure on $\cL$ and an extension of $i$ to an $L_\infty$ quasi-isomorphism.
\end{lm}
\begin{proof}
We define an anchor on $\cL$ by the composite $ai$ where $a$ is the anchor on $\cM$.  To define the brackets, we use the homotopy transfer theorem for $L_\infty$ algebras (e.g.~\cite[Theorem 10.3.9]{Loday2012}).  It allows us to transfer the $\KK$-linear $L_\infty$ algebra structure on $\cM$ to one on $\cL$, and to extend $i$ to a sequence of  maps
\[
i_n\colon \underbrace{\cL \times \cdots \times \cL}_{n\textrm{ times}}  \to 
 \cM
 \]
giving a $\KK$-linear  $L_\infty$-quasi-isomorphism.

By inspection of the explicit formulae for the transferred structure, we see that the maps $i_n$ are $\cO(U)$-multilinear due to the equations $h^2 = 0$ and $hi=0$.  The binary bracket is
\[
[x, y]_{\cL} = p[ix,iy]_{\cM}
\]
and since $pi=\id$ we see that it satisfies the correct Leibniz rule. Finally, the higher brackets on $\cL$ are manifestly $\cO$-linear since those of $\cM$ are.
\end{proof}

\subsection{Globalization}

We now give a definition of $L_\infty$ algebroids valid for arbitrary manifolds $X$.  The idea is to glue together $L_\infty$ algebroids defined on affine subsets as above.  Crucially, we are allowed to glue by arbitrary quasi-isomorphisms, rather than just strict isomorphisms.  This presents the difficulty that the cocycle condition may be satisfied only up to higher homotopies which in turn satisfy an infinite sequence of coherences.  This coherences may be succinctly summarized as follows.

Suppose that  $U$ is an affine manifold.  Since the simplicial category of $L_\infty$ algebroids on $U$ is enriched in Kan complexes, its simplicial nerve is an $\infty$-category. We denote by $\LA(U)$ the maximal $\infty$-subgroupoid, i.e.~the subcategory where we only keep invertible morphisms ($L_\infty$-quasi-isomorphisms).  

Given an open embedding of affine manifolds $i\colon U\rightarrow V$, there is an obvious pullback functor
\[
i^*\colon \LA(V) \to \LA(U)
\]
given by restricting bundles, anchors, brackets, etc. 
\begin{defn}
Let $X$ be a manifold. The \defterm{space of $L_\infty$ algebroids on $X$} is the homotopy limit 
\[
\LA(X) = \lim_{U\subset X} \LA(U).
\]
over affine open subsets $U \subset X$.
\end{defn}

Notice that if $X$ is an affine manifold, the category of affine open subsets has a final object; hence this definition recovers the original definition.  One can show that $\LA$ satisfies descent, i.e.~it is an $\infty$-sheaf. This allows one to describe an $L_\infty$ algebroid on a general manifold $X$ by choosing an affine cover $\{U_i\}\rightarrow X$, putting an $L_\infty$ algebroid $\cL_i$ on each $U_i$, a quasi-isomorphism $g_{ij}\colon \cL_i \to \cL_j$ on each double overlap $U_i \cap U_j$, a homotopy $h_{ijk}\colon g_{ij}g_{jk}g_{ki} \Rightarrow 1$ on each triple overlap $U_i \cap U_j \cap U_k$, etc., satisfying an approriate cocycle condition.   We will not use this statement in the paper

\begin{remark}If $X$ is a complex manifold or smooth algebraic variety that is not quasi-projective, then an $L_\infty$ algebroid on $X$ may not have an underlying global complex of vector bundles, essentially because coherent sheaves on $X$ may not have global resolutions by vector bundles.  However, if $\cL$ is a Lie 1-algebroid, it is defined by a global vector bundle, giving a Lie algebroid in the classical sense. \qed
\end{remark}

Given an $L_\infty$ algebroid on $X$, we obtain a quotient stack
\[
\pi\colon X \to \XL
\]
as in the affine case.  The derived global sections (\v{C}ech cohomology) of the structure sheaf $\cO_\XL$ are modelled by the cdga
\[
\mathbb{R}\Gamma(\XL,\cO_\XL) = \lim_{U \subset X} \cO([U/\cL|_U]).
\]
over affine open submanifolds of $X$. 

\begin{ex}
If $\cL=0$ is the trivial Lie algebroid, then
\[
\mathbb{R}\rsect(\XL,\cO_\XL) = \C^{\bullet}(X,\cO_X)
\]
is the complex of derived global sections of the structure sheaf (i.e.~the \v{C}ech cohomology).  \qed
\end{ex}

\begin{ex}
More generally, if $\cL$ is a Lie 1-algebroid, then there is no possibility of higher homotopies, and so the Chevalley--Eilenberg complexes on affines assemble into the usual sheaf of cdgas $\wedge^\bullet \cL^\vee$ on $X$.  In this case 
\[
\mathbb{R}\rsect(\XL,\cO_\XL) \cong \C^\bullet(X,\wedge^\bullet \cL^\vee)
\]
is the hypercohomology complex defining the Lie algebroid cohomology. \qed
\end{ex}

\section{Geometry of the quotient}
\label{sec:quotient-forms}
\subsection{Quasi-coherent sheaves}

Let $U$ be an affine manifold and let $\cL$ an $L_\infty$ algebroid over $U$. A complex of quasi-coherent sheaves on $\UL$ corresponds to a complex $\cE_0$ of quasi-coherent sheaves on $U$ that carries an $\cL$-action.  This action can equivalently be described by a differential on $\cE_0 \otimes \cO(\UL)$, making it into a dg module over the Chevalley--Eilenberg algebra; this is the approach of ``representations up to homotopy''~\cite{Abad2012}.  As is well known, there are many complexes (such as the tangent complex) that are only non-canonically isomorphic to a representation up to homotopy.  This motivates considering slightly more general objects, as follows.

Consider the ideal
\[
\cI = \ker \pi^* \subset \cO(\UL)
\]
of positive degree elements, and the induced filtration by subcomplexes
\[
\cO(\UL) \supset \cI \supset \cI^2\supset\cdots.
\]
The associated graded complex is canonically isomorphic to the symmetric algebra $\Sym(\cL^\vee[-1])$, but with differential induced by the original differential $\delta$ on the complex of vector bundles $\cL$, and all higher brackets forgotten.

Now let $\cE = (\cE^\bullet,\delta)$ be a dg module over $\cO(\UL)$. It inherits a filtration
\[
\cE \supset  \cI\cdot \cE \supset \cI^2 \cdot \cE \supset \cdots,
\]
and we denote by
\[
\pi^* \cE = \cE/\cI\cE \cong \cE \tens_{\cO(\UL)} \cO(U)
\]
the complex of $\cO(U)$-modules obtained by taking the zeroth graded piece.

\begin{defn}
A dg module $\cE$ is \defterm{quasi-coherent complex on $\UL$} if the natural map
\[
\pi^* \cE \tens_{\cO(U)} \cI^n/\cI^{n+1} \to \cI^n\cE/\cI^{n+1} \cE
\]
is an isomorphism for all $n$.  In this case, the complex $\pi^*\cE$ is called the \defterm{pullback of $\cE$ to $U$}.
\end{defn}

The point of the definition is that the dg-module $\cE$ becomes a representation up to homotopy once we make a non-canonical choice of splitting of the filtration.  We say that a morphism $\cE \rightarrow \cE'$ of dg-modules is a \defterm{filtered quasi-isomorphism} if it induces a quasi-isomorphism on associated gradeds.  We denote by $\QCoh(\UL)$ the $\infty$-category whose objects are quasi-coherent complexes on $\UL$, and whose morphisms are given by localizing the usual morphisms of dg modules at the filtered quasi-isomorphisms. 

The pullback defined above gives a functor
\[
\pi^* : \QCoh(\UL) \to \QCoh(U)
\]
and the following useful result is an immediate consequence of the definition:
\begin{lm}
Let $\cL$ be an $L_\infty$ algebroid on $U$, and suppose that
\[
\phi \colon  \cE \to \cE'
\]
is a morphism in $\QCoh(\UL)$.  Then $\phi$ is a weak equivalence if and only if the pullback
\[
\pi^*\phi \colon \pi^*\cE \to \pi^*\cE'
\]
is a quasi-isomorphism of complexes of $\cO(U)$-modules.
\label{lm:quasiiso}
\end{lm}

Evidently quasi-coherent complexes may be pulled back along open embeddings of affine manifolds.  Thus if $\cL$ is an $L_\infty$ algebroid on an arbitrary manifold $X$, we may define
\[
\QCoh(\XL) = \lim_{U \subset X} \QCoh([U/\cL|_U])
\]
where the limit is taken over all affine open submanifolds. Thus an object $\cE \in \QCoh(\XL)$ corresponds, on affine open set $U$, to  a dg module $\cE|_U$ over the Chevalley--Eilenberg algebra, and these modules are glued by filtered quasi-isomorphisms in a homotopy coherent way.

For a general morphism of $L_\infty$ algebroids $f \colon \cM \to \cL$ over bases $Y$ and $X$, there is a pullback functor
\[
f^* \colon \QCoh(\XL) \to \QCoh(\YM)
\]
which is given on affine subsets $U \subset X$ and $V \subset Y$ by the usual formula
\[
f^* \cE|_V = \cE|_U \tens_{\cO(\UL)} \cO(\VM).
\]

\subsection{The tangent and cotangent complexes}
\label{sec:tangent}
Suppose that  $\cL$ is an $L_\infty$ algebroid on an affine manifold $U$. The \defterm{tangent complex of $[U/\cL]$} is the module $\cT_{[U/\cL]}$ of graded derivations of $\cO(\UL)$, equipped with the differential defined by taking the commutator of derivations with the Chevalley--Eilenberg differential $\delta$.  See \cite{Crainic2008} for a detailed discussion in the case of Lie 1-algebroids, where the complex is called the deformation complex and shifted in degree by one from our convention.

Viewing elements of $\cL$ as derivations on $\cO(\UL)$ by interior contraction gives an inclusion $\cL[1] \to \cT_\UL$.  Meanwhile, restricting derivations to the degree zero component of $\cO(\UL)$ gives a projection $\cT_\UL \to \cT_U$.  Extending these maps linearly, we obtain an exact sequence
\begin{align}
\xymatrix{
0 \ar[r] & \cO(\UL) \utens{\cO(U)}  \cL[1] \ar[r] & \cT_\UL \ar[r] & \cO(\UL) \utens{\cO(U)} \cT_U \ar[r] & 0 
}\label{eqn:tan-exact}
\end{align}
of graded $\cO(\UL)$-modules that may by split by picking a connection on $\cL$.   We note that the submodule $\cO(\UL)\otimes \cL[1]$ is not a subcomplex; in general, it carries no natural differential.

From this exact sequence, we immediately obtain the following
\begin{prop}
\label{prop:tan-pullback}
The dg $\cO(\UL)$-module $\cT_\UL$ is quasi-coherent, and its pullback  along the projection $\pi\colon U \to [U/\cL]$ is given by
\begin{align*}
\pi^*\cT_{[U/\cL]} &= \rbrac{\xymatrix{\cdots \ar[r]^-{\d} &
\cL_1 \ar[r]^-{\d} & \cL_0 \ar[r]^-{a} & \cT_U
}}
\end{align*}
where $\cT_U$ sits in degree zero and $\cL_i$ sits in degree $-(i+1)$.
\end{prop}

\begin{remark}
This result has the usual geometric interpretation: at a point $p \in U$,  the zeroth cohomology of the fibre
\[
\coH^0(\pi^*\cT_U|_p) = T_pU/a(\cL_0|_p)
\]
is the normal space to the $\cL$-orbit of $p$, i.e.~the Zariski tangent space to the projection $\pi(p) \in \UL$.  Meanwhile, the negative cohomologies form a graded Lie algebra that represents the stabilizer of $p$ under the $\cL$-action.\qed
\end{remark}

The \defterm{cotangent complex of $\UL$} is similarly defined as the  dual $\cO(\UL)$-module $\cT_\UL^\vee$ of K\"ahler differentials.  The de Rham differential is given by the universal derivation
\[
\ddr\colon \cO(\UL) \to \coT_{\UL}
\]
which is, by construction, a morphism of complexes: $\d \ddr = \ddr \d$.

As for $\cO(\UL)$, the module $\coT_\UL$ has a bigrading by weight and degree.  Geometrically, it is the degree grading on $\coT_{\UL}$ that is more fundamental: the differential on $\coT_\UL$ has degree one,  and  it corresponds to the \v{C}ech cohomology $\coH^\bullet(\UL,\cT_\UL^\vee)$.

The weight grading, on the other hand, is not preserved by the differential.  Nevertheless, it is convenient when one wants explicit formulae.  Dualizing \eqref{eqn:tan-exact} and taking the weight-$n$ piece, we obtain the sequence
\begin{align}
\xymatrix{
0 \ar[r] & \wedge^n \cL^\vee \otimes \coT_U \ar[r] & (\coT_\UL)^{\bullet,n} \ar[r] & \wedge^{n-1}\cL^\vee \otimes \cL^\vee[-1]  \ar[r] & 0 
}\label{eqn:cotan-exact}
\end{align}
which has the following differential-geometric interpretation~\cite{AriasAbad2011,Li-Bland2015}:

\begin{prop}\label{prop:operators}
The weight-$n$ subspace of $\coT_\UL$ is canonically isomorphic to the $\cO(U)$-module of pairs $(\omega_n,\overline \omega_n)$, consisting of a first-order totally skew-symmetric polydifferential operator
\[
\omega_{n}\colon \underbrace{\cL \times \cdots \times \cL}_{n\textrm{ times}} \to \Omega^1(U)
\]
and a tensor
\[
\bomega_n \in \Sym^{n-1}(\cL^\vee[-1]) \otimes \cL[-1]
\]
related by the symbol equation
\begin{align}
\omega_n(x_1,\ldots,x_{n-1},fx_n) = f\omega_n(x_1,\ldots,x_n) + \bomega_n(x_1,\ldots,x_{n-1}|x_n) \cdot \ddr f \label{eqn:operator-symbol}
\end{align}
Here $\bomega_n(x_1,\ldots,x_{n-1}|x_n)$ denotes the canonical $\cO(U)$-linear pairing of $\bomega_n$ with $x_1x_2\cdots x_{n-1} \otimes x_n$.
\end{prop}

\begin{proof}
The weight-$n$ subspace is evidently spanned by monomials of the form $u_1\cdots u_n\, \ddr f$ and $u_1\cdots u_{n-1}\, \ddr u_n$ with $f \in \cO(U)$ and $u_i \in \cL^\vee$.  We simply explain how to assign a pair $(\omega_n,\bomega_n)$ to such monomials, and leave to the reader the straightforward check that the map gives a well-defined isomorphism, e.g.~that is compatible with the fundamental relation
\[
\ddr(fu) = (\ddr f)u+f(\ddr u)
\]
for K\"ahler differentials.

Firstly, for the monomial $u_1\cdots u_n\,\ddr f$, we set $\bomega_n = 0$, so that the operator $\omega_n\colon \cL^{\times n} \to \Omega^1(U)$ must be $\cO(U)$-multilinear.  We then define $\omega_n$ by interpreting $\omega$ as a monomial in $\Sym^n(\cL^\vee[-1]) \otimes \Omega^1(U)$.

Secondly, for a monomial $\omega = u_1\cdots u_{n-1} \, \ddr u_n$ we set
\[
\bomega_n = u_1\cdots u_{n-1} \otimes u_n
\]
and define the operator $\omega_n$ by the formula
\begin{align*}
\omega_{n}(x_1,\ldots,x_n) &= \sum_{i=1}^{n}(-1)^{\epsilon} \abrac{ u_1\cdots u_{n-1} , x_1\cdots \widehat{x_i}\cdots x_n} \ddr \abrac{u_n,x_i}
\end{align*}
where the sign $\epsilon$ is determined by the Koszul sign rule.
\end{proof}

Given an element $\omega \in \coT_\UL$, we write $\omega = (\omega_n,\bomega_n)_{n \ge 0}$ to indicate the sequence of pairs obtained by applying \autoref{prop:operators} to all of the weight components of $\omega$.  The differential and de Rham derivative can now be described explicitly in terms of brackets, extending the formulae for classical Lie algebroids~\cite{Abad2012}. We describe the idea and leave the verification of the formulae  as an exercise to the reader.

Firstly, if  $u \in \cO(\UL)$, we use the construction in \autoref{prop:operators} to deduce that the weight-$n$ part of its de Rham differential $\ddr u \in \coT_\UL$ is
\begin{align*}
(\ddr u)_n(x_1,\ldots,x_n) &= \ddr (u_n(x_1,\ldots,x_n)) \\
\overline{(\ddr u)}_n(x_1,\ldots,x_{n-1}|x_n) &= u_n(x_1,\ldots,x_n).
\end{align*}
Secondly, the fact that $\ddr$ is a morphism of complexes results in the following formula for the differential of $\omega = (\omega_n,\bomega_n)$:
\begin{align*}
(\d \omega)_n(x_1,\ldots,x_n) &=  (\dCE \omega)(x_1,\ldots,x_n) \\
&\ \ \ + \sum_{i=1}^n (-1)^\epsilon \lie{a x_i} \omega_{n-1}(x_1,\ldots,\widehat{x_i},\ldots,x_n) \\
\overline{(\d \omega)}_n(x_1,\ldots,x_{n-1}|x_n) &= (\dCE \overline{\omega})(x_1,\ldots,x_{n-1}|x_n) \\
&\ \ \ + \sum_{i=1}^{n-1} (-1)^\epsilon \lie{ax_i} \bomega_{n-1}(x_1,\ldots,\widehat{x_i},\ldots,x_{n-1}|x_n) \\
&\ \ \ + (-1)^\epsilon\iota_{ax_n}\omega_{n-1}(x_1,\ldots,x_{n-1}).
\end{align*}
where $\dCE \omega$ is defined by the same formula as in  \eqref{eqn:CEterm}, and $\dCE \bomega$ is defined similarly, but with the additional constraint that $x_n$ always appears at the end, i.e.~we sum over terms of the form $\overline{\omega}_j([x_{\sigma_1},\ldots,x_{\sigma_i}],x_{\sigma_{i+1}},\ldots,x_{\sigma_{n-1}}|x_n)$ and terms of the form  $\overline{\omega}_j(x_{\sigma_1},\ldots,x_{\sigma_{j-1}}|[x_{\sigma_{j}},\ldots,x_{\sigma_{n-1}},x_n])$.

\subsection{Differential forms}
\label{sec:forms}
With the cotangent complex in hand, we can now describe the algebra of differential forms. Let $\cL$ be an $L_\infty$ algebroid on an affine manifold $U$.  Then the $p$-forms on $[U/\cL]$ are given by $p$th exterior power
\[
\Omega^{p}(\UL) = \wedge^p\, \cT^\vee_{[U/\cL]} = \rbrac{\Sym^p_{\cO(\UL)}\rbrac{\cT^\vee_{[U/\cL]}[-1]}}[p]
\]
Thus the $p$-forms are a complex in their own right.  We write $\Omega^{p,q}(\UL)$ for the $p$-forms of degree $q$.  In particular, we have
\[
\Omega^{p,0}(\UL) = \Omega^p(U),
\]
and the projection
\[
\pi^*\colon \Omega^{p,\bullet}(\UL) \to \Omega^p(U)
\]
models the pullback of forms along the quotient map $\pi\colon U \to \UL$.

As always, the internal differential
\[
\delta\colon \Omega^{p,\bullet}(\UL) \to \Omega^{p,\bullet+1}(\UL)
\]
should be seen as analogous to the  \v{C}ech differential for the sheaf of $p$-forms on the quotient stack $\UL$.
\begin{remark}
Suppose that $\cL$ is a Lie 1-algebroid and $\hat{\cG}_\bullet$ is the nerve of the formal groupoid integrating $\cL$. Then one can identify $\coH^q([U/\cL], \Omega^p)$ with $\coH^q(\Omega^p(\hat{\cG}_\bullet))$. For a \emph{Lie} groupoid $\cG$, there is a natural ``van Est'' differentiation map~\cite{AriasAbad2011,Mehta2006} from $\Omega^p(\cG_\bullet)$ to $\Omega^{p,\bullet}(\UL)$, and it is shown in \cite{AriasAbad2011} that this map is an isomorphism on cohomology, under certain connectivity assumptions on the source fibres.  (We remark that $\Omega^{i,j}(\UL)$ was denoted by $W^{j,i}$ in that paper.)  Since the fibers of a formal groupoid $\hat{\cG}$ are, in a sense, contractible, we expect this identification to hold for arbitrary formal groupoids.
\qed
\end{remark}

The de Rham differential on functions extends to a morphism of complexes
\[
\ddr\colon \Omega^{p,\bullet}(\UL) \to \Omega^{p+1,\bullet}(\UL),
\]
in the usual way.  Thus $\Omega^{\bullet,\bullet}(\UL)$ is a bigraded bidifferential algebra, playing the role of the full Hodge diamond of $\UL$.

\begin{defn}
Let $X$ be an arbitrary manifold, and let $\cL$ be an $L_\infty$ algebroid on $X$.  The \defterm{algebra of differential forms on $\XL$} is the bigraded bidifferential algebra
\[
\Omega^{\bullet,\bullet}(\XL) = \lim_{U\subset X} \Omega^{\bullet,\bullet}([U/\cL|_U])
\]
obtained by taking the limit over all affine open subsets of $X$.
\end{defn}

\subsection{Closed forms}

Unlike forms on a manifold, for which being closed is a property, in the derived or stacky settings, closure is extra data: we ask for the $p$-form $\omega$ to satisfy the equation $\ddr \omega = 0$ only up to higher homotopies.  This condition is  phrased most succinctly by analogy with the  Poincar\'e lemma.  On a $C^\infty$ manifold $X$, the sheaf of closed differential forms of degree at least $p$ has a natural resolution by acyclic sheaves, given by its inclusion  in the complex $(\Omega^{\ge p}_X,\ddr)$.  In the algebraic setting, or on a quotient $\XL$, the Poincar\'e lemma will typically fail, but we may declare that closed $p$-forms are described by the total complex
\[
\Omega^{\ge p}(\XL) = \Tot\rbrac{\xymatrix{
\Omega^{p,\bullet}(\XL) \ar[r]^-{\ddr} & \Omega^{p+1,\bullet}(\XL)  \ar[r]^-{\ddr} &  \cdots
}
}
\]
\begin{defn}
Let $X$ be a manifold and $\cL$ an $L_\infty$ algebroid on $X$. A \defterm{closed $(p,q)$-form on $\XL$} is a cocycle
\[
\omega \in \Z^{p+q} \Omega^{\ge p}(\XL)
\]
\end{defn}

Thus a closed $(p,q)$-form consists of whole sequence of forms
\[
\omega_p,\omega_{p+1},\omega_{p+2},\ldots
\]
where $\omega_{p+j} \in \Omega^{p+j,q-j}(\XL)$, and these data satisfy the equations
\begin{align*}
\d\omega_p &= 0 \\
\ddr \omega_p + \d \omega_{p+1} &= 0 \\
\ddr \omega_{p+1} + \d \omega_{p+2}&= 0 \\
&\ \  \vdots
\end{align*}

Since closed $(p,q)$-forms are cocycles in a complex, there is a natural notion of homotopy equivalence between them, given by  coboundaries.  Then there are homotopies of homotopies, etc., so that closed forms naturally form a higher groupoid:
\begin{defn}\label{def:space-of-closed}
The \defterm{space of closed $(p,q)$-forms on $\XL$} is the simplicial set $|\Omega^{\ge 2,q}(\XL)|$  assigned to the truncated complex $\tau^{\le p+q} \Omega^{\ge p}(\XL)$ by the Dold--Kan correspondence.
\end{defn}

Similar to the one-form case described in \autoref{prop:operators}, one can describe arbitrary forms by decomposing them into weight components that are tensors and differential operators.  But giving a complete description the space of closed $(p,q)$-forms in this decomposition is quite cumbersome: for example, the closed $(2,2)$-forms that we will focus on later have nine distinct components, satisfying  15 different equations.  One must then account for their higher homotopies, which give several more components and equations.

Fortunately, it turns out that most of the information in a closed form is actually redundant.   For example, \cite[Proposition 4.12]{Li-Bland2015} can be interpreted as saying that that the pullback of forms gives an isomorphism of the de Rham cohomologies
\[
\pi^*\colon \coH^\bullet_{\textrm{dR}}(\XL) \to \coH^\bullet_{\textrm{dR}}(X),
\]
This is consistent with the idea that the quotient map $\pi\colon X \to \XL$ expresses $\XL$ as a formal neighbourhood of $X$; hence they have the same topology.  We now explain how to extend this approach to give an efficient model for the closed $p$-forms.

It is enough to describe the construction when $\cL$ is an $L_\infty$ algebroid on an affine manifold $U$.   Consider the canonical Euler derivation $\xi \in \cT^{0}_\XL$ that multiplies a homogeneous element of $\cO(\UL)$ by its degree.  It gives a homotopy operator
\[
h\colon \Omega^{\bullet,\bullet}(\UL) \to \Omega^{\bullet-1,\bullet}(\UL)
\] 
via the interior contraction
\[
h(\omega) = \begin{cases}
\tfrac{1}{q}\iota_\xi \omega & \omega \in \Omega^{\bullet,q}(\UL), q>0 \\
0 & \omega \in \Omega^{\bullet,0}(\UL) = \Omega^\bullet(U).
\end{cases}
\]
Then using the Cartan formula $\lie{\xi} = \ddr \iota_\xi + \iota_\xi \ddr$, we see that every strictly $\ddr$-closed element of $\Omega^{\bullet,>0}(\UL)$ is actually $\ddr$-exact.  More precisely, we may use the homotopy to define the \defterm{complex of potentials}
\[
\Pot^{p-1,\bullet}(\UL) = \img h \subset \Omega^{p-1,\bullet}(\UL) 
\]
with differential defined by
\begin{align}
\dpot = h \d \ddr = -h \ddr \d = (\ddr h  - 1) \d. \label{eqn:dpot}
\end{align}
Then $h$ and $\ddr$ give mutually inverse isomorphisms between the complex of potentials, and the complex of strictly closed $p$-forms of positive degree.

There is also a natural \defterm{twisting map}
\[
\twist\colon \Omega^{\ge p}(U) \to \Pot^{p-1}(\UL)[1]
\]
defined as follows.  Given an arbitrary element $G \in \Omega^{p+q}(U)$, we may contract it with the  exterior power $\wedge^{q+1}a$ of the anchor $a \in \cL_0^\vee \otimes \cT_U$ to obtain the element
\[
\twist G = \iota_{\wedge^{q+1} a} G \in \wedge^{q+1}\cL_0^\vee \otimes \Omega^{p-1}
\]
which we view as a $(p-1)$-form whose coefficient lies in $\wedge^{q+1} \cL_0^\vee \subset \cO(\UL)$.
\begin{defn}
The  \defterm{normalized complex of closed $p$-forms} is the complex
\[
\Omcl^{p}(\UL) = \Pot^{p-1}(\UL) \oplus \Omega^{\ge p}(U)
\]
with differential given by
\[
\dtw = \begin{pmatrix}
\dpot & \twist \\
0 & \ddr
\end{pmatrix}\colon  \xymatrix{{\begin{matrix}\Pot^{p-1}(\UL) \\ \oplus \\ \Omega^{\ge p}(U)\end{matrix}}\ar[r] & {\begin{matrix}\Pot^{p-1}(\UL) \\ \oplus \\ \Omega^{\ge p}(U).\end{matrix}}}
\]
\end{defn}
Thus the normalized complex fits in an exact sequence
\[
\xymatrix{
0 \ar[r] & \Pot^{p-1}(\UL) \ar[r] & \Omcl^{p}(\UL) \ar[r] & \Omega^{\ge p}(U) \ar[r] & 0.
}
\]
Although the normalized complex contains no elements in $\Omega^{p+j,q-j}(\UL)$ for $0 < j < q$, it still captures the full complexity of  homotopy closed forms:
\begin{thm}\label{thm:reduced}
The normalized complex so-defined is, indeed, a complex.  Moreover, there is a canonical homotopy equivalence
\[
\Omega^{\ge p}(\UL) \cong \Omcl^{p}(\UL)
\]
compatible with the projections to $\Omega^{\ge p}(U)$.
\end{thm}

\begin{proof} Consider first the case in which the differential and the $L_\infty$ algebroid structure on $\cL$ are trivial, so that $\delta$, $\delta_\Pot$ and $\twist$ are all zero.  Then the theorem holds.  Indeed, our discussion above shows that the homotopy operator $h$ gives a special deformation retract
\[
\xymatrix{
 (\Omega^{\ge p}(\UL),\ddr)\ar@/^2pc/[rr]^-{p} & & (\Pot^{p-1}(\UL) \oplus \Omega^{\ge p}(U), \d_0) \ar@/^2pc/[ll]^{i}.
}
\]
where the differential $\d_0$ on the right acts only on $\Omega^{\ge p}(U)$, where it is given by the de Rham differential.  The  projection $p$ is induced by the homotopy $h\colon \Omega^{p}(\UL) \to \Pot^{p-1}(\cL)$ and the projection $\Omega^{\ge p}(\UL) \to \Omega^{\ge p}(U)$.  Meanwhile the inclusion $i$ is induced by $\ddr \colon \Pot^{p-1}(\cL) \to \Omega^{p}(\UL)$ and the inclusion $\Omega^{\ge p}(U) \to \Omega^{\ge p}(\UL)$.  

We now consider the general case as a perturbation of this one.  That is, given a nontrivial $L_\infty$ algebroid structure, we view the  total differential $\d+\ddr$ on $\Omega^{\ge p}(\UL)$ as a perturbation of the de Rham differential $\ddr$.  By the Homological Perturbation Lemma~\cite{Brown1965,Gugenheim1972} (see also \cite{Crainic2004a}), the operation
\[
\delta' = \d_0 + p(1-\d\,h)^{-1}\d i
\]
defines a differential on $\Omcl^p(\UL)$, so that the projection
\[
p' = p(1+(1-\d\,h)^{-1}\d h)\colon \Omega^{\ge p}(\UL)  \to \Omcl^{p}(\UL)
\]
and the inclusion
\[
i'  = (1+h(1-\d  h)^{-1}\d) i\colon  \Omcl^{p}(\UL) \to  \Omega^{\ge p}(\UL)
\]
continue to give a special deformation retract. 

 Since $h$ acts by zero on $\Omega^{\bullet}(U)$, it is clear that $p'$ intertwines the projections of $\Omcl^p(\UL)$ and $\Omega^{\ge p}(\UL)$ to $\Omega^{\ge p}(U)$.  Thus to prove the theorem,  it suffices to verify that the differential $\delta'$ is precisely the twisted differential $\dtw$ described above.  To see this, suppose first that $\alpha \in \Pot^{p-1}(\UL)$.  Considering the bidegrees, we find
\begin{align*}
\delta' \alpha &= \delta_0\alpha + p(1-\d h)^{-1}\d \ddr \alpha \\
&= 0 + (h \d + (h \d)^2 + \cdots ) \ddr \alpha \\
&= h  \d \ddr \alpha \\
&= \dpot \alpha 
\end{align*}
so that $\delta$ and $\dtw$ agree on $\Pot^{p,\bullet}(\cL)$.

Meanwhile, if  $G \in \Omega^{p+q}(U)$, bidegree considerations give
\begin{align*}
\delta G &= \ddr G + p(1+\d h + (\d h)^2+\cdots)\d G \\
&= \ddr G + p (\d h)^q \d G \\
&= \ddr G + (h \d)^{q+1} G,
\end{align*}
where we have used the fact that $p$ acts as the homotopy $h$ on $\Omega^p(\UL)$.   We claim that the operator $(h\d)^{q+1}$  is precisely the twisting map $\twist$.  Indeed, considering the definition of $\twist$, it is enough by induction to show the operator $h\d$ acts on the subspace
\[
\wedge^\bullet\cL_0^\vee \otimes \Omega^\bullet(U) \subset \cO(\UL) \cdot \Omega(U) \subset \Omega(\UL)
\]
by  wedging and contracting with the anchor $a \in \cL_0^\vee\otimes \cT_U$.  But this is straightforward: given an element $u \in \cO(\UL)$ and $f_1,\ldots,f_k \in \cO(U)$, we compute
\begin{align*}
\d( u\, \ddr f_1 \, \cdots \, \ddr f_k) &= \d u\, \ddr f_1\cdots \ddr f_k \\
&\ \ \ + \sum_{j=1}^k (-1)^{\epsilon} u\,\ddr(\d f_j)\ddr f_1\cdots\widehat{\ddr f_j}\cdots\ddr f_k.
\end{align*}
using the fact that $\d $ and $\ddr$ commute.  By definition, $h$ annihilates the first term completely.  Meanwhile, we have the identity
\[
h\ddr(\d f_j) = \d f_j = a^\vee(\ddr f_j)
\]
relating the differential and the anchor.   We conclude that
\begin{align*}
h\d( u\, \ddr f_1 \, \cdots \, \ddr f_k) &=  \sum_{j=1}^k (-1)^{\epsilon} (u \cdot  a^\vee(\ddr f_i))\ddr f_1\cdots\widehat{\ddr f_j}\cdots\ddr f_k
\end{align*}
is the contraction with the anchor, as desired.
\end{proof}

The operator $\dpot = \ddr h \d-\d$ can be written explicitly in terms of its weight components by combining the formula for $\delta$ with an explicit formula for $\ddr h$.  We shall only need the action of $\ddr h$ on one-forms in the paper:
\begin{ex}
Consider a monomial $\alpha = u_1 \cdots u_j \ddr v$ with $u_1,\ldots,u_j,v \in \cL^\vee[-1]$.  Using the fact that the Euler derivation has total degree $-1$, we see that
\[
h \alpha =  \frac{(-1)^{|\alpha|-|v|}}{|\alpha|}   u_1\cdots u_j \cdot v.
\]
Applying $\ddr$ to this expression, and converting it back into operators as in \autoref{prop:operators}, we get the following formula for the action of $\ddr h$ on an arbitrary element $\alpha \in \Omega^1(\UL)$, not just monomials:
\begin{align*}
\overline{(\ddr h \alpha)}_n(x_1,\ldots,x_{n-1}|x_n) &= \frac{1}{|\alpha|}  \sum_{i=1}^{n} (-1)^\epsilon \overline{\alpha}_n(x_1,\ldots,x_{i-1},x_n,x_{i+1},\ldots,x_{n-1}|x_i) 
\end{align*}
and
\begin{align*}
(\ddr h \alpha)_n(x_1,\ldots,x_n) = \frac{1}{|\alpha|} \ddr\rbrac{ \sum_{i=1}^{n} (-1)^\epsilon \overline{\alpha}_n(x_1,\ldots,x_{i-1},x_n,x_{i+1},\ldots,x_{n-1}|x_i) }
\end{align*}
where the Koszul sign is determined by treating $|$ as a degree one symbol.

Thus the effect of the operator $\ddr h$ on $\Omega^1(\UL)$ is to apply an appropriate symmetrization to the tensorial part, and then adjust the operator component by an exact term that has the correct symbol. \qed
\end{ex}

\section{Shifted symplectic structures}

\subsection{Shifted symplectic forms}

The notions of shifted symplectic and Lagrangian structures on derived stacks~\cite{Pantev2013} can now be adapted to our context. Given a cocycle $\omega \in \Z^q( \Omega^{2,\bullet}(\XL))$, i.e.~a global two-form of degree $q$, we obtain a morphism
\[
\omega\colon \cT_\XL \to \coT_\XL[q],
\]
by interior contraction.  We say that $\omega$ is \defterm{nondegenerate} if this map is a quasi-isomorphism, i.e.~an isomorphism in $\QCoh(\XL)$.

\begin{defn}
Let $X$ be a manifold, and let $\cL$ be an $L_\infty$ algebroid on $X$. A \defterm{$q$-shifted symplectic structure on $\XL$} is a closed $(2,q)$-form
\[
\omega \in \Z^{2+q}\Omega^{\ge 2}(\XL)
\]
whose underlying two-form is nondegenerate in the above sense. 
\end{defn}

Let us describe the nondegeneracy condition more explicitly for affine manifolds $U$.   By \autoref{lm:quasiiso} and \autoref{prop:tan-pullback}, a two-form $\omega$ is nondegenerate if and only if its pullback  induces a quasi-isomorphism of complexes of vector bundles on $U$:
\begin{align}
\begin{gathered}
\xymatrix{
\pi^*\cT_\UL \ar[d]^{\pi^*\omega} &
\cdots \ar[r] &  \cL_{q-1} \ar[r]^{\d}\ar[d]^{\omega}   & \cL_{q-2} \ar[r]^{\d}\ar[d]^{\omega} & \cdots \ar[r]^-{a} & \cT_U \ar[r] \ar[d]^{\omega}& 0\\
 \pi^*\coT_\UL[q] &
0 \ar[r] & \coT_U \ar[r]^-{a^\vee} & \cL_{0}^\vee \ar[r]^-{\d} &\cdots \ar[r]^-{\d} &  \cL_{q-1}^\vee \ar[r] & \cdots
 }
\end{gathered} 
 \label{eqn:omegapullback}
\end{align}
Here the vertical maps are obtained by picking out appropriate tensorial components from the weight decomposition of $\omega$.

Notice that the top complex  is bounded on the right, while the bottom complex is bounded on the left.  The existence of a quasi-isomorphism therefore puts an obvious bound on the amplitude of $\cL$:
\begin{lm}\label{lm:symp-ampl}
Suppose that $\UL$ admits a $q$-shifted symplectic structure for $q > 0$.  Then the natural truncation map $\cL \to \tau^{>(-q)}\cL$ is a quasi-isomorphism of complexes of vector bundles.  Hence $\cL$ is equivalent to a Lie $(q-1)$-algebroid.
\end{lm}
\begin{proof} 
By homotopy transfer (\autoref{thm:homotopytransfer}), it is enough to show that we have a quasi-isomorphism of complexes of vector bundles.  The argument is standard: let us denote the truncation by
\[
\cL' = \rbrac{\xymatrix{
0 \ar[r]&  \cL_{q-1}/\delta \cL_q \ar[r] & \cL_{q-2} \ar[r] & \cdots \ar[r] & \cL_0
}}.
\]
Considering the quasi-isomorphism \eqref{eqn:omegapullback}, we see that the cohomology of $\cL$ vanishes in degree less than $-(q-1)$, and hence the natural projection $\cL \to \cL'$ is a quasi-isomorphism of complexes of $\cO(U)$-modules.  It remains to see that $\cL_{q-1}/\delta \cL_q$ is actually a vector bundle (a projective module), which is equivalent to $\mathrm{Ext}^i(\cL_{q-1}/\delta \cL_q,-) = 0$ for $i > 0$.  This vanishing follows easily from the above quasi-isomorphism and the fact that $\cL_{q-2},\ldots,\cL_0$ are vector bundles.  
\end{proof}

 For a fixed $L_\infty$ algebroid $\cL$ on a manifold $X$, the \defterm{space of $q$-shifted symplectic forms} is the full simplicial subset
\[
\symp_q(\XL) \subset |\Omega^{\ge 2,q}(\XL)|
\]
whose zero-simplices are $q$-shifted symplectic forms, where $|\Omega^{\ge 2,q}(\XL)|$ is the space of closed two-form from \autoref{def:space-of-closed}.

Symplectic structures may be pulled back along $L_\infty$ quasi-isomorphisms, so that the assignment $\cL \mapsto \symp_q(\XL)$ is functorial.   Hence the simplicial sets $\symp_q(\XL)$ for varying $\cL$ may be assembled to give a single $\infty$-groupoid:
\begin{defn}
The \defterm{space of $q$-shifted symplectic algebroids on $X$} is the simplicial set $\Symp_q(X)$ obtained by applying the Grothendieck construction to the functor $\symp_q\colon \LA(X)^{op} \to \SSet$.
\end{defn}

 Thus an object of the $\infty$-groupoid  $\Symp_q(X)$ is an $L_\infty$ algebroid $\cL$ equipped with a shifted symplectic form $\omega$. Meanwhile, a morphism  $(\cL,\omega) \to (\cL',\omega')$ is an $L_\infty$-quasi-isomorphism $f\colon \cL \to \cL'$ together with a coboundary that trivializes the cocycle $f^*\omega' - \omega \in \Omega^{\ge 2,q}(\XL)$, and similarly for higher homotopies.

More explicitly, suppose $X$ is affine. Then an $n$-simplex in $\Hom(\cL, \cL')$ is represented by a morphism of $\Omega_n$-linear commutative dg algebras \[\cO([X/\cL'])\otimes \Omega_n\to \cO([X/\cL])\otimes \Omega_n.\] Passing to $\Omega_n$-linear de Rham complexes, we obtain an $\Omega_n$-linear morphism \[\Omega^{\geq 2}([X/\cL'])\otimes \Omega_n\to \Omega^{\geq 2}([X/\cL])\otimes \Omega_n\]
and the required data is a homotopy between the pullback of $\omega'\otimes 1$ and $\omega\otimes 1$.

\subsection{Examples}

We now give some simple examples of Lie algebroids equipped with shifted symplectic structures.

\subsubsection{Zero-shifted symplectic algebroids}
\label{sec:0shift}
Arguing as in \autoref{lm:symp-ampl}, we easily see that a zero-shifted symplectic algebroid must be quasi-isomorphic to a single vector bundle concentrated in degree zero, hence a classical Lie algebroid $\cL$.

The only piece of data underlying a zero-shifted symplectic form on $\XL$ is a two-form $B \in \Omega^{2,0}(\XL) = \Omega^2(X)$.  By \autoref{thm:reduced}, the closure conditions amount to the equations
\begin{align*}
\ddr B &= 0 \in \Omega^3_X & 
\twist B = \iota_a B &= 0 \in \cL^\vee \otimes \Omega^1_X,
\end{align*}
and the nondegeneracy condition is that we have a quasi-isomorphism
\begin{align}
\begin{gathered}
\xymatrix{
\pi^*\cT_\XL \ar[d]^{\pi^*\omega} &  \cL\ar[d] \ar[r]^-{a} & \cT_X \ar[r]\ar[d]^-{B}& 0\ar[d] \\
 \pi^*\coT_\XL & 0 \ar[r] & \coT_X \ar[r]^-{a^\vee} & \cL^\vee .
 }
 \end{gathered}
 \label{eqn:0shift-iso}
\end{align}
 Considering the cohomology in degree one, we see that $a^\vee$ must be surjective; equivalently, the anchor map embeds $\cL$ in $\cT_X$ as the tangent bundle of a regular foliation.  Then the condition $\iota_a B = 0$ means that $B \in \wedge^2 (\cT_X/\cL)^\vee \subset \Omega^2(X)$ is a two-form in the directions transverse to the foliation.  Moreover, considering the cohomology of \eqref{eqn:0shift-iso} in degree zero, we see that the nondegeneracy of the zero-shifted symplectic structure is equivalent to the nondegeneracy of $B$ in the transverse directions.
  
Finally, the condition $\ddr B = 0$ means that $B$ is closed in the transverse directions, and also that it is invariant along the foliation.  Thus we conclude that a zero-shifted symplectic Lie algebroid on $X$ is simply a regular foliation of $X$, equipped with an invariant transverse symplectic structure, i.e.~a classical symplectic structure on the leaf space  $\XL$. 

\subsubsection{Transitive shifted symplectic algebroids}
\label{sec:symplectic-transitive}
Now consider a Lie algebroid $\cL$ on a manifold $X$, and assume that $\cL$ that is \defterm{transitive}, i.e.~its anchor map is surjective.  Thus $\cL$ fits into an exact sequence
\[
\xymatrix{
0 \ar[r] & \g \ar[r] & \cL \ar[r] &\cT_X \ar[r] &  0,
}
\]
where $\g$ is a vector bundle equipped with an $\cO_X$-linear Lie bracket.

There is a natural adjoint action of $\cL$ on $\g$, and hence $\g$ descends to a complex on $\XL$, which we denote by $\g_\XL$.  On an affine open subset $U$, it is represented by the dg module $\cO(\UL) \tens_{\cO(U)} \g$ over $\cO(\UL)$.

From the quasi-isomorphism
\[
\pi^*\cT_\XL \cong \rbrac{\xymatrix{\cL \ar[r] & \cT_X}} \cong (\xymatrix{\g \ar[r] & 0}) = \g[1]
\]
and the exact sequence \eqref{eqn:tan-exact}, we see that the natural inclusion
\[
\g_\XL[1] \to \cT_\XL
\]
is a quasi-isomorphism of complexes on $\XL$.  (See also~\cite[Corollary 4]{Crainic2008}.)

Thus the quasi-isomorphism $\pi^*\cT_\XL \to \pi^*\coT_\XL[q]$ induced by a $q$-shifted symplectic structure gives an equivalence 
\[
\g[1] \cong \g^\vee[q-1],
\]
of complexes on $X$, which forces $q = 2$.  Thus the algebra of differential forms on $\XL$ is canonically identified with $\Sym(\g^\vee_\XL[-2])$.  In particular, the space of closed $(2,2)$-forms on $\XL$ is given by the discrete set
\[
\coH^0(\XL,\Sym^2(\g^\vee_\XL)) = \coH^0(X,\Sym^{2}(\g^\vee))^\cL,
\]
of $\cL$-invariant symmetric bilinear forms on $\g$.  In conclusion, we have the following classification:
\begin{prop}\label{prop:2-shift-transitive}
Let $\cL$ be a transitive Lie 1-algebroid on $X$, and suppose that the kernel $\g \subset \cL$ of its anchor is nontrivial.  Then the only symplectic structures on $\XL$ have shift two, and they are in bijective correspondence with nondegenerate $\cL$-invariant symmetric bilinear forms on $\g$.
\end{prop}

The pullback of forms along the projection $X \to \XL$ gives a natural map
\[
\coH^0(X,\Sym^{2}(\g^\vee))^\cL \to \coH^2(X,\Omega^{\ge 2}_X).
\]
This connecting homomorphism may be computed by covering $X$ with affine open subsets $U_i$ on which the inclusion $\g \subset \cL$ can be split, and using the resulting projections $\cL|_{U_i}\to \g|_{U_i}$ to give an explicit cocycle representative in $\Omega^{\ge 2}([U_i/\cL|_{U_i}])$.  Comparing splittings on double and triple overlaps allows one to extend this to a cocycle on all of $X$.

We shall not give the details of this process here; let us simply state the result in a special case.   Suppose that $G$ is a Lie group equipped with a nondegenerate pairing $\abrac{-,-}$ on its Lie algebra, and $P$ is a principal $G$-bundle on $X$.  Its Atiyah algebroid $\cL$ is an extension
\[
\xymatrix{
0\ar[r]& \ad P \ar[r] & \cL \ar[r] & \cT_X \ar[r] & 0
}
\]
and the kernel $\g = \ad P$ inherits an invariant nondegenerate pairing from $\abrac{-,-}$, producing a two-shifted symplectic structure on $[X/\cL]$.  In this case, the pullback of the symplectic form gives the class in $\coH^2(X,\Omega^{\ge 2}_X)$ associated with $P$ and $\abrac{-,-}$ by Chern--Weil theory, namely the first Pontryagin class~\cite{Bressler2007,Severa1998--2000}.   If we think of $P$ as a map $X \to BG$, then $[X/\cL]$ is a model for the formal neighbourhood of $X$ in $BG$ with its two-shifted symplectic structure~\cite[p.~299--300]{Pantev2013}.

\subsection{Isotropic and Lagrangian structures}

Suppose that $f\colon (Y,\cM) \to (X,\cL)$ is a morphism of manifolds equipped with $L_\infty$ algebroids.  Suppose further that $\XL$ carries a $q$-shifted symplectic structure $\omega \in \symp_q(\XL)$.  Then we may ask if the induced morphism
\[
f\colon \YM \to \XL
\]
defines a Lagrangian in $\XL$.  As with the definition of closed forms, the notion of Lagrangian corresponds to extra data on the map, rather than a property.

\begin{defn}
An \defterm{isotropic structure} on the map $f\colon \YM \to \XL$ is a choice of coboundary for the cocycle $f^*\omega \in \Omega^{\ge 2}(\YM)$.
\end{defn}

Picking out appropriate weight components of an isotropic structure, we obtain null homotopy of the composite sequence
\begin{equation}
\xymatrix{
\cT_{\YM} \ar[r] & f^*\cT_\XL \ar[r]^-{\omega} & \coT_\YM[q]
}
\label{eq:isotropicsequence}
\end{equation}
and hence a morphism
\[
\cN_f \to \coT_\YM[q]
\]
where $\cN_f = \Cone(\cT_\YM \to f^*\cT_\XL)$ is the normal complex.

\begin{defn}
An isotropic structure is \defterm{Lagrangian} if the induced morphism $\cN_f \to \coT_\YM[q]$ is a quasi-isomorphism, or equivalently, \eqref{eq:isotropicsequence} is a fibre sequence of complexes.
\end{defn}

\begin{ex}\label{ex:lag-quotient}
Consider the case $X=Y$ and $\cM=0$, so that $f$ is simply the quotient map
\[
f = \pi\colon X \to \XL
\] 
From the isomorphism
\[
\pi^*\cT_\XL \cong \rbrac{\xymatrix{\cdots \ar[r] & \cL_1\ar[r] & \cL_0 \ar[r]^{a} &\cT_U}}
\] 
we see that $\cN_\pi \cong \cL[1]$, so that the isotropic structure induces a morphism $\cL[1] \to \coT_X[q]$.  If the quotient map is Lagrangian, then $\cL \cong \coT_X[q-1]$.  For $q>1$ the only possibility is that $\cL$ is abelian (i.e.~the anchor and brackets vanish), so this condition is quite restrictive.  But as we recall in \autoref{sec:shift1-lagrangian}, the case $q=1$ is nontrivial: it gives the Lie algebroid $\coT_X$ associated to a Poisson structure. \qed
\end{ex}

\begin{ex}
Suppose that $\cL$ carries a zero-shifted symplectic structure, so that it is defined by a regular foliation equipped with a transverse symplectic form $B \in \wedge^2 (\cT_X/\cL)^\vee$ as in \autoref{sec:0shift}.  We claim that Lagrangians in $\XL$ correspond to immersed  Lagrangians in the leaf space of the foliation, in the sense of classical symplectic geometry.

Indeed, given a map $f\colon Y \to X$, consider a Lagrangian $\YM \to \XL$ obtained by lifting $f$ to a Lie algebroid morphism $\cM \to \cL$.  Since the symplectic structure has degree zero, there is no room for homotopies between forms.  Thus being Lagrangian is a condition, rather than extra data.

Computing the normal complex of $\YM \to \XL$, we see that the map is Lagrangian if and only if $B$ induces a quasi-isomorphism
\begin{align*}
\xymatrix{
\pi^*\cN \ar[d] & \cM \ar[r]\ar[d] & \cT_Y \ar[d]\ar[r] & f^*(\cT_X/\cL) \ar[r]\ar[d]^-{f^*B} & 0\ar[d] \\
\pi^*\coT_\YM & 0 \ar[r] & 0\ar[r] & \coT_Y \ar[r] & \cM^\vee.
}
\end{align*}
Considering the cohomology in degree 1 and $-2$, we see that $\cM$ must embed in $\cT_Y$ as an involutive subbundle, giving a regular foliation of $Y$.  Then, from the cohomology in degree $-1$, we see that $\cT_Y/\cM$ embeds in $f^*(\cT_X/\cL)$ as a subsheaf.  Finally, considering the cohomology in degree zero, we see that this subsheaf must be a subbundle $\cT_Y/\cM\subset f^*(\cT_X/\cL)$ that is maximally isotropic with respect to $B$.  Hence the map $\YM \to \XL$ is a Lagrangian immersion of the leaf spaces, as claimed. \qed
\end{ex}

\begin{ex}
Let $G$ be a Lie group equipped with an nondegenerate invariant bilinear form on its Lie algebra, and let $H \subset G$ be a closed subgroup whose corresponding Lie subalgebra is Lagrangian.  If $P$ is a principal $G$-bundle on $X$ with Atiyah algebroid $\cL(P)$, then $[X/\cL(P)]$ is 2-shifted symplectic as in \autoref{sec:symplectic-transitive}.  Moreover, if $P$ admits a reduction of structure to a principal $H$-bundle $P'$, then natural inclusion of Atiyah algebroids $\cL(P') \subset \cL(P)$ gives a Lagrangian map $[X/\cL(P')] \to [X/\cL(P)]$. \qed
\end{ex}

\section{Two-shifted symplectic forms}
\label{sec:shift2}

\subsection{Twisted Courant algebroids}
\label{sec:courant}
We now turn to the classification of shifted symplectic structures of low degree.  The strategy is to use homotopy transfer and the normalized complex of closed two-forms to reduce the complicated data of a shifted symplectic $L_\infty$-algebroid to a normal form in terms of the following objects:
\begin{defn}[\cite{Hansen2010,Liu1997}]\label{def:affineTCA}
Let $U$ be an affine manifold.  A \defterm{twisted Courant algebroid} on $U$ is a tuple $(\cE,K,\abrac{-,-},\circ,a)$, where
\begin{itemize}\setlength{\itemsep}{0em}
\item  $K \in \Omcl^4(U)$ is a global closed 4-form
\item $\cE$ is a locally free sheaf, i.e.~a vector bundle
\item $\abrac{-,-} \in \Sym^2(\cE^\vee)$ is a nondegenerate symmetric bilinear pairing 
\item $a\colon \cE \to \cT_U$ is an $\cO_U$-linear map, called the anchor, and
\item $\cour{-,-} \colon \cE\times \cE\rightarrow \cE$ is a bilinear operator, called the Courant--Dorfman bracket.
\end{itemize}
These data are subject to the following equations concerning their action on sections $x,y,z \in \cE$:
\begin{align}
\cour{x , f y} &= f\cour{x , y} + (\lie{ax}f) y \label{eq:tcalgebroid1} \\
\cour{x , x} &= \tfrac{1}{2}a^*\ddr \abrac{x,x} \label{eq:tcalgebroid2} \\
\lie{ax} \abrac{y,z} &= \abrac{ \cour{x,y} , z } + \abrac{ y, \cour{x,z}} \label{eq:tcalgebroid3} \\
\cour{x,\cour{ y , z}} &= \cour{\cour{x, y}, z} + \cour{ y,  \cour{ x, z} }
- \tfrac{1}{2}a^*\iota_{ax}\iota_{ay}\iota_{az} K, \label{eq:tcalgebroid4}
\end{align}
where $a^* : \Omega^1_U \to \cE$ is the transpose of the anchor with respect to $\abrac{-,-}$.
\end{defn}
\begin{defn} A twisted Courant algebroid is a \defterm{Courant algebroid} if its four-form is trivial: $K=0$.
\end{defn}

We shall often suppress the anchor, bracket and pairing in the notation, and simply say that $\cE$ or $(\cE,K)$ is a twisted Courant algebroid.  If we wish to emphasize that a twisted Courant algebroid is a Courant algebroid, we may refer to it as ``untwisted''.

\begin{ex}The original example of a Courant algebroid~\cite{Courant1988,Courant1990,Dorfman1987} is the bundle $\cE = \cT_U \oplus \cT_U^\vee$.  Its anchor is the projection to $\cT_U$, and its pairing is the  canonical one induced by the duality of $\cT_U$ and $\cT_U^\vee$.  Finally, the bracket is defined by
\[
\cour{x+\alpha,y+\beta} = [x,y]+\lie{x} \beta -  \iota_y \ddr \alpha 
\] 
where $x,y\in \cT_U$ and $\alpha,\beta \in \cT_U^\vee$.  This Courant algebroid is called the \defterm{standard Courant algebroid on $U$}.  \qed
\end{ex}

\begin{ex}\label{ex:h-twist-std}
Given a three-form $H \in \Omega^3(U)$, we can modify the bracket on the standard Courant algebroid by setting
\[
\cour{x + \alpha,y + \beta}_H = \cour{x+\alpha,y+\beta}+ \iota_{x}\iota_{y}H.
\]
We then obtain a twisted Courant algebroid, with four-form $K = \ddr H$. \qed
\end{ex}
We will give further examples in \autoref{sec:shift2-ex}.
Twisted Courant algebroids on an affine manifold $U$ naturally form a 2-groupoid $\TCAlgd(U)$, defined as follows:
\begin{itemize}
\item[] \textbf{Objects} of $\TCAlgd(U)$ are twisted Courant algebroids $(\cE,K)$ on $U$.
\item[] \textbf{1-Morphisms} $(\cE,K) \to (\cE',K')$ are pairs $(g,H)$, where $g\colon \cE\rightarrow \cE'$ is an orthogonal bundle isomorphism that is compatible with the anchors, and $H \in \Omega^3(U)$ is a three-form that relates the brackets and four-forms:
\begin{align*}
g\cour{x,y} - \cour{gx,gy}' &= \tfrac{1}{2}a'^*\iota_{ax}\iota_{ay} H \\
K' - K &= \ddr H
\end{align*} 
\item[] \textbf{2-morphisms} $(g,H) \Rightarrow (\widetilde{g},\widetilde{H})$  are two-forms $B\in \Omega^2(U)$ such that
\begin{align*}
\widetilde{g}-g &= \tfrac{1}{2} a'^* Ba \\ 
\widetilde{H} - H &= \ddr B.
\end{align*}
\end{itemize}

Courant algebroids on $U$, in contrast, are more rigid: they do not admit differential forms as higher symmetries, and therefore form a 1-groupoid $\CAlgd(U)$:
\begin{itemize}
\item[] \textbf{Objects} of $\CAlgd(U)$ are Courant algebroids $\cE$ on $U$
\item[] \textbf{1-Morphisms} $\cE \to \cE'$ in $\CAlgd(U)$ are given by bundle isomorphisms that preserve the pairing, anchor and bracket.
\end{itemize}

For an inclusion $U' \subset U$ of affine manifolds, there are obvious restriction functors $\TCAlgd(U) \to \TCAlgd(U')$ and $\CAlgd(U) \to \CAlgd(U')$, obtained by pulling back bundles and forms. This allows us to define the space of (twisted) Courant algebroids on an arbitrary manifold $X$ by gluing along open covers:
\begin{align*}
\TCAlgd(X) &= \lim_{U \subset X} \TCAlgd(U) & \CAlgd(X) &= \lim_{U\subset X} \CAlgd(U)
\end{align*}
where the limit is taken over the category of all affine subsets of $X$.

For (untwisted) Courant algebroids, the result of such a gluing is evident: since the isomorphisms between Courant algebroids are strict isomorphisms of vector bundles that preserve all of the structure, an object of the 1-groupoid $\CAlgd(X)$ is just a global vector bundle $\cE$ on $X$, equipped with an anchor, a pairing, and a bracket on its sheaf of sections, satisfying the axioms of \autoref{def:affineTCA} with $K=0$.  The notion of morphisms is the same as in the affine case.

But for twisted Courant algebroids, the situation is more complicated, due to the presence of 2-morphisms.  With respect to an affine open covering $\{U_i\}$ of $X$, a twisted Courant algebroid $\cE$ on $X$ is described by the following data:
\begin{itemize}\setlength{\itemsep}{0em}
\item A twisted Courant algebroid $(\cE_i,K_i)$ on each affine open set $U_i$, as above
\item A 1-morphism $(g_{ij},H_{ij})\colon \cE_i \to \cE_j$ on every double overlap $U_i \cap U_j$
\item A 2-morphism $B_{ijk}\colon (g_{ij}g_{jk}g_{ki},H_{ij}+H_{jk}+H_{ki}) \Rightarrow (\id_{\cE_i},0)$ on every triple overlap $U_i\cap U_j \cap U_k$,
\end{itemize}
subject to an appropriate cocycle condition. 

A key feature of twisted Courant algebroids is that the vector bundle gluing maps $g_{ij}$ satisfy the twisted cocycle condition
\begin{align}
g_{ij} g_{jk}g_{ki} = 1+\tfrac{1}{2}a_i^* B_{ijk} a_i  \in \Hom_{U_{ijk}}(\cE_i , \cE_i). \label{eqn:bundle-twist}
\end{align}
The cocycle $B_{ijk} \in \Z^2(X,\Omega^2_X)$ defines an $\Omega^2$-gerbe on $X$, and when this gerbe is nontrivial, $\cE$ will typically be a twisted bundle, rather than a global vector bundle in the classical sense.

More generally, the differential form data associated to a twisted Courant algebroid give a cocycle
\[
(B_{ijk},H_{ij},K_i) \in \Z^2(X,\Omega^{\ge 2}_X)
\]
in the hypercohomology of the truncated de Rham complex.  We call its cohomology class the \defterm{twisting class of $\cE$}:
\[
[\cE] = [(B_{ijk},H_{ij},K_i)] \in \coH^2(X,\Omega^{\ge 2}_X).
\]
since it is the obstruction to finding an untwisted Courant algebroid that is equivalent to $\cE$.

\subsection{Classification of two-shifted symplectic structures}

\label{sec:2shift-obj}

In this section we explain how to reduce an arbitrary two-shifted symplectic algebroid to a normal form, given in terms of twisted Courant algebroids.  More precisely, we will prove the following
\begin{thm}\label{thm:shift2}
For any manifold $X$, there is a canonical equivalence
\[
\Symp_2(X) \cong \TCAlgd(X)
\]
between the $\infty$-groupoid of two-shifted symplectic $L_\infty$ algebroids on $X$ and the 2-groupoid of twisted Courant algebroids.  Under this correspondence, the class in $H^2(X,\Omega^{\ge 2}_X)$ determined by the pullback of a two-shifted symplectic structure agrees with the twisting class of the corresponding twisted Courant algebroid.
\end{thm}

Considering the definitions, it is enough to prove the theorem for an arbitrary affine manifold $U$.  The strategy of the proof is to consider an auxiliary 2-groupoid $\TCAConn(U)$ and produce a canonical pair of equivalences
\[
\xymatrix{
& \TCAConn(U) \ar[rd]^\sim \ar[ld]_\sim & \\
\Symp_2(U) && \TCAlgd(U).
}
\]
The 2-groupoid $\TCAConn(U)$ has a simple description: its objects are  pairs $(\cE,\nabla)$, where $\cE$ is a twisted Courant algebroid and $\nabla$ is a metric connection, i.e.~a connection on the vector bundle $\cE$ that preserves the nondegenerate pairing.  Morphisms in $\TCAConn(U)$ are morphisms of the underlying twisted Courant algebroids; the connections do not play a role.

The equivalence $\TCAConn(U) \to \TCAlgd(U)$ is provided by the forgetful functor.  Indeed, this functor is fully faithful by definition, and it is essentially surjective because every principal bundle on an affine manifold admits a connection.   Thus the rest of the section is concerned with the construction of the equivalence $\TCAConn(U) \cong \Symp_2(U)$. 

\subsubsection{Objects}

We begin by describing the equivalence on the level of objects.   Note that by \autoref{lm:symp-ampl}, an arbitrary two-shifted symplectic $L_\infty$ algebroid is equivalent to a Lie 2-algebroid
\[
\cL = \rbrac{\xymatrix{
\cL_1 \ar[r] & \cL_0
}}
\] 
The symplectic structure on $\cL$ is determined by a potential $\beta \in \Pot^{1,2}(\UL)$ and a closed four-form $K \in \Omcl^4(U)$ such that $\dpot \beta = \twist K$.  Now $\beta$ is a $(1,2)$-form, and by degree considerations, its weight decomposition consists of operators
\begin{align*}
\phi = \beta_1 \colon \cL_1 &\to \Omega^1(U) &
\psi = \beta_2 \colon \cL_0\times \cL_0 &\to \Omega^1(U)
\end{align*}
and their symbols
\begin{align*}
\bbeta_1 &\in \cL_1^\vee & Q = \bbeta_2 &\in \cL_0^\vee \otimes \cL_0^\vee
\end{align*}
The condition $h\beta = 0$ for $\beta$ to define an element of $\Pot^{1,2}(\UL)$ is equivalent to setting $\bbeta_1 = 0$ identically, and requiring $Q$ to be symmetric. Thus $\phi$ is $\cO(U)$-linear, while  
\[
\psi(x,fy) = f\psi(x,y) + Q(x,y) \ddr f
\]
for $x,y \in \cL_0$ and $f \in \cO(U)$.

Applying the relation between the differential $\d$ and the $L_\infty$ algebroid structure, we see that closure equation $\dpot \alpha = \twist K$ is equivalent to the following four equations obtained from the weight and degree decomposition of $\dpot \alpha$:
\begin{align}
&\lie{ax}\psi(y, z) + \psi(x, [y, z]) - \tfrac{1}{3}\ddr(Q(x, [y, z]) +  \iota_{ax}\psi(y, z)) + \circlearrowright   + \phi([x, y, z]) \nonumber \\& = \iota_{ax}\iota_{ay}\iota_{az} K \label{eq:2sympalgd1} \\
&\lie{ax} \phi(u) - \phi([x, u]) - \psi(x, \delta u) + \ddr Q(x, \delta u) = 0 \label{eq:2sympalgd2} \\
&3\lie{ay} Q(x, z)- 2Q(x, [y, z]) + \iota_{ax} \psi(y, z) - (x\leftrightarrow y) \nonumber \\
&+ 4Q([x, y], z) - 2\iota_{az} \psi(x, y) = 0 \label{eq:2sympalgd3} \\
&2Q(\d u,x)  + \iota_{ax} \phi(u) = 0 \label{eq:2sympalgd4}
\end{align}
where $x,y,z \in \cL_0$ and $u \in \cL_1$.

The last equation simply says that the pullback is a morphism of complexes
\begin{align}
\begin{gathered}
\xymatrix{
\pi^*\cT_\UL \ar[d] & \cL_1 \ar[r]^-{\delta} \ar[d]^-{\phi} & \cL_0 \ar[r]^-{a}\ar[d]^-{\tfrac{1}{2}Q} & \cT_U \ar[d]^{\phi^\vee} \\
\pi^*\coT_\UL[2] & \coT_U \ar[r]^-{a^\vee} & \cL_0^\vee \ar[r]^-{\delta^\vee} & \cL_1^\vee
}
\end{gathered}\label{eqn:2-shift-tan0}
\end{align}

This diagram may be simplified in the following way:

\begin{prop}\label{lm:courantstrict}
Any two-shifted symplectic $L_\infty$ algebroid is symplectically quasi-isomorphic to one for which the diagram \eqref{eqn:2-shift-tan0} has the form
\begin{align}
\begin{gathered}
\xymatrix{
\pi^*\cT_\UL \ar[d] & \coT_U \ar[r]^-{\delta} \ar@{=}[d] & \cE \ar[r]^-{a} \ar[d]^-{\tfrac{1}{2}\abrac{-,-}} & \cT_U \ar@{=}[d] \\
\pi^*\coT_\UL[2] & \coT_U \ar[r]^-{a^\vee} & \cE^\vee \ar[r]^{\delta^\vee} & \cT_U,
}
\end{gathered}\label{eqn:2-shift-tan}
\end{align}
where $\abrac{-,-} \in \Sym^2(\cE^\vee)$ is a nondegenerate symmetric bilinear form.
\end{prop}

\begin{proof}
Consider the complex $\widetilde{\cL} = \cL\oplus \cT^\vee_U\oplus \cT^\vee_U[1]$ with the differential twisted by $\phi$ and the identity $\cT^\vee_U\rightarrow \cT^\vee_U$ as follows:
\[
\xymatrix{
\cT^\vee_U\oplus \cL_1 \ar^{  \begin{psmallmatrix} \id & \phi \\ 0 & \d \end{psmallmatrix}}[rr] && \cT^\vee_U\oplus \cL_0
}
\]

Let us also define $\widetilde{\cE} = \cL\oplus \cT^\vee_U$ with the differential twisted by $\phi$. The natural projection $p\colon \widetilde{\cL}\rightarrow \cL$ has a splitting $i\colon \cL\rightarrow \widetilde{\cL}$ given by
\[
\xymatrix{
\cL_1 \ar^{(-\phi, \id)}[d] \ar[r] & \cL_0 \ar^{\id}[d] \\
\cT^\vee \oplus \cL_1 \ar[r] & \cT^\vee \oplus \cL_0
}
\]
making $p$ into a deformation retract. Therefore, by \autoref{lm:homotopytransfer1} we obtain an $L_\infty$ algebroid structure on $\widetilde{\cL}$; moreover, pulling back the two-shifted symplectic structure on $\cL$ along $p$ we obtain a two-shifted symplectic structure $\widetilde{\omega}$ on $\widetilde{\cL}$ of the following shape:
\[
\xymatrix{
\cT^\vee_U\oplus \cL_1 \ar^{(0, \phi)}[d] \ar[r] & \cT^\vee_U\oplus \cL_0 \ar^{ \tfrac{1}{2} \begin{psmallmatrix}0 & 0 \\ 0 & Q\end{psmallmatrix}}[d] \ar[r] & \cT_U \ar^{(0, \phi)}[d] \\
\cT^\vee_U \ar[r] & \cT_U\oplus \cL_0^\vee \ar[r] & \cT_U\oplus \cL_1^\vee
}
\]

We have a subspace $\cT_U \otimes \coT_U \subset \widetilde{\cL}_0^\vee\otimes \cT^\vee_U\subset \Pot^{1, 1}(\widetilde{\cL})$ and it contains a canonical element $\tau$ corresponding to the identity. The form $\widetilde{\omega}+\dtw\tau$ is still nondegenerate and it has the following shape:
\[
\xymatrix{
\cT^\vee_U\oplus \cL_1 \ar^{(\id,0)}[d] \ar[r] & \cT^\vee_U\oplus \cL_0 \ar^{ \tfrac{1}{2}\begin{psmallmatrix} 0 & a \\ a^\vee & Q \end{psmallmatrix}}[d] \ar[r] & \cT_U \ar^{(\id, 0)}[d] \\
\cT^\vee_U \ar[r] & \cT_U\oplus \cL_0^\vee \ar[r] & \cT_U\oplus \cL_1^\vee
}
\]
The nondegeneracy of the two-shifted symplectic structure on $\cL$ is now equivalent to the morphism
\[
{
\begin{psmallmatrix}
0 & a \\
 a^\vee & Q
\end{psmallmatrix}} \colon
\xymatrix{
\widetilde{\cE} \ar[r] & \widetilde{\cE}^\vee
}
\]
being a quasi-isomorphism. But $\widetilde{\cE}$ is a complex of vector bundles concentrated in non-positive degrees, so the projection $\widetilde{\cE}\rightarrow \cH^0(\widetilde{\cE})$ is also a quasi-isomorphism. Therefore, we can replace $\widetilde{\cE}$ by its cohomology $\cE=\cH^0(\tilde{\cE})$ on which the pairing is strictly nondegenerate and the claim follows.
\end{proof}

We say that a two-shifted symplectic algebroid is in \defterm{Courant form} if it has the form described in \autoref{lm:courantstrict}.  Given an algebroid in Courant form, we may use the nondegeneracy of the pairing to define a connection $\nabla \colon \cE\rightarrow \Omega^1(U) \otimes \cE$ and a bracket $\cour{-,-} \colon \cE \times \cE \to \cE$  by the formulae
\begin{align}
 \langle\nabla x, y\rangle &= \tfrac{1}{2}(\ddr \langle x, y\rangle - \psi(x, y)) \label{eq:2sympconn}
\\
\abrac{ \cour{x,y},z} &= \abrac{ [x, y], z } + \abrac{ \nabla_{az} x, y }\nonumber \\
&= \abrac{ [x, y], z } + \tfrac{1}{2}\lie{az}\abrac{ x, y } - \tfrac{1}{2}\iota_{az}\psi(x, y)\nonumber
\end{align}
or equivalently
\begin{align}
\cour{x, y} = [x,y] + \tfrac{1}{2}a_\cE^* \ddr \abrac{x,y} - \tfrac{1}{2} a_{\cE}^* \psi(x,y).\label{eq:2sympcourant}
\end{align}
Because of the skew-symmetry of $\psi$, the connection $\nabla$ is automatically metric,~i.e. it satisfies the equation
\[
\ddr \langle x, y\rangle = \langle\nabla x, y\rangle + \langle x, \nabla y\rangle.
\]
The following result then describes the equivalence $\TCAConn(U) \to \Symp_2(U)$ on the level of objects:

\begin{prop}\label{prop:symp-cour}
The formulae \eqref{eq:2sympconn} and \eqref{eq:2sympcourant} give a bijective correspondence between shifted symplectic $L_\infty$ algebroids in Courant form and twisted Courant algebroids equipped with a metric connection.
\end{prop}

\begin{proof}
For an algebroid in Courant form, we have that $\phi = \id$.  Thus the closure conditions \eqref{eq:2sympalgd1}, \eqref{eq:2sympalgd2} and \eqref{eq:2sympalgd4} uniquely determine the triple bracket $\cL_0\times \cL_0\times \cL_0\rightarrow \cL_1$, the binary bracket $\cL_0\times \cL_1\rightarrow \cL_1$ and the differential $\delta$ in terms of the remaining data.  It is therefore sufficient to see that the axioms for a twisted Courant algebroid are the same as the remaining equations for the symplectic $L_\infty$ algebroid structure.

Indeed, axiom \eqref{eq:tcalgebroid1} for a twisted Courant algebroid is equivalent to the Leibniz rule for the $L_\infty$ bracket $[-, -] \colon \cE \times \cE \to \cE$, axiom \eqref{eq:tcalgebroid2} is equivalent to the antisymmetry of the bracket $[-, -]$, and axiom \eqref{eq:tcalgebroid4} is equivalent to the Jacobi rule for the $L_\infty$ brackets on $\cL$.  Finally, axiom \eqref{eq:tcalgebroid3} is equivalent to the equation
\[
\lie{a(x)} Q(y, z) = Q([x, y], z) + \tfrac{1}{2}\lie{a(y)}Q(x, z) - \tfrac{1}{2}\iota_{a(y)} \psi(x, z) + (y\leftrightarrow z),
\]
which is the symmetrization of the remaining closure equation \eqref{eq:2sympalgd3}. Conversely, if \eqref{eq:tcalgebroid3} is satisfied, the left-hand side of \eqref{eq:2sympalgd3} is completely antisymmetric, but its antisymmetrization is obviously zero. Therefore, axiom \eqref{eq:tcalgebroid3} is equivalent to \eqref{eq:2sympalgd3}.
\end{proof}

\subsubsection{$1$-morphisms}

Suppose we are given a pair $(\cL,\omega)$ and $(\cL',\omega')$ of symplectic algebroids in Courant form, corresponding to twisted Courant algebroids $(\cE,K)$ and $(\cE',K')$.  An $L_\infty$ morphism $g\colon \cL \to \cL'$ consists of a quasi-isomorphism of complexes
\[
g\colon \cL \to \cL'
\]
that preserves the anchors, and a map
\[
\tg\colon \wedge^2 \cE \to \cT^\vee_U = \cL'_1,
\]
satisfying the $L_\infty$ 
morphism equations
\begin{align}
g[x,y] - [gx,gy] &= \d \tg(x,y) \label{eq:2symp1mor1} \\
g[x,u] - [gx,gu] &= \tg(\xi, \d u) \label{eq:2symp1mor2} \\
g[x,y,z] - [gx,gy,gz] &= -\tg([x,y],z) + \tg([x,z],y) - \tg([y,z],x)  \nonumber \\
&\ \ \ + [\tg(x,y), gz] - [\tg(x,z), gy]  +[\tg(y,z), gx] \label{eq:2symp1mor3}
\end{align}
for $x,y,z \in \cE$ and $u \in \cT^\vee_U$.

To extend such a quasi-isomorphism to a symplectic equivalence, we must include a homotopy of closed two-forms, given by an element of 
\[
\Omcl^{2,1}(\UL) = \Pot^{1,1}(\UL) \oplus \Omega^3(U).
\]
The elements of $\Pot^{1, 1}(\UL)$ are $(1,1)$-forms in the image of $h$.  Considering the weight decomposition, it is easy to see that in fact $\Pot^{1, 1}(\UL)\cong \cE^\vee \otimes \Omega^1(U)$.  The homotopy then consists of elements
\begin{align*}
\tau &\in \cE^\vee \otimes \Omega^1(U) & H \in \Omega^3(U)
\end{align*}
satisfying the homotopy equation $g^*\omega'-\omega = \dtw(\tau + H)$, which gives the system
\begin{align}
K' - K &= dH \label{eq:2symp1mor4} \\
gu - u &= -\tau(\delta u) \label{eq:2symp1mor5} \\
\abrac{gx,gy} - \abrac{x,y} &= \tfrac{1}{2}(\iota_{ax} \tau y + \iota_{ay} \tau x) \label{eq:2symp1mor6} \\
-\tilde{g}(x, y) + \psi'(gx, gy) - \psi(x, y) &= \tau[x, y] - \lie{ax} \tau y + \lie{ay} \tau x \nonumber \\
&\ \ \ \ + \tfrac{1}{2} \ddr(\iota_{ax} \tau y - \iota_{ay} \tau x) + \iota_{a x}\iota_{a y } H \label{eq:2symp1mor7}
\end{align}

Observe that the equation \eqref{eq:2symp1mor2} follows from \eqref{eq:2symp1mor7}. Similarly, equation \eqref{eq:2symp1mor3} follows from the definition of the triple bracket $[-, -, -]$ given by \eqref{eq:2sympalgd1}. Equation \eqref{eq:2symp1mor5} determines the morphism $g\colon \cL \rightarrow \cL'$ in degree $-1$ and equation \eqref{eq:2symp1mor7} determines $\tilde{g}(x, y)$. 

We conclude that a 1-morphism in $\Symp_2(X)$ is uniquely determined by the triple $(g, \tau, H)$, where $g\colon \cE \rightarrow \cE'$ is bundle map, $\tau \in \cE^\vee \otimes \Omega^1(U)$ and $H \in \Omega^3(U)$ satisfy the equations \eqref{eq:2symp1mor1}, \eqref{eq:2symp1mor4} and \eqref{eq:2symp1mor6}.

We say that a 1-morphism is in \defterm{Courant form} if $\tau = 0$.  In this case, the equations reduce to the equations for $(g,H)$ to give a 1-morphism of twisted Courant algebroids $(\cE,K) \to (\cE',K')$.  In this way, we define the functor $\TCAConn(U) \to \Symp_2(U)$ on the level of 1-morphisms. 

\subsubsection{$2$-morphisms}
\label{sec:2shift-2mor}

Finally, suppose we are given a pair of 1-morphisms $f_i\colon \cL \to \cL'$ for $i=1,2$ determined by the data $g_i\colon \cE \to \cE'$, $\tau_i \in \cE^\vee \otimes \Omega^1(U)$ and $H_i \in \Omega^3(U)$ as above.

A 2-morphism $f_1 \Rightarrow f_2$ in $\Symp_2(U)$ consists of a homotopy operator on the complexes, i.e.~an $\cO_U$-linear map 
\[
h\colon \cE \to \cT^\vee_U = \cL'_1
\]
and a form 
\[
B \in \Omcl^{2,0}(\UL) = \Omega^2(U)
\]
We require  $(\d + \ddr)B$ to equal the difference of the 2-form data appearing in $f_i$, giving the equations
\begin{align}
H_2 - H_1 &= \ddr B & \tau_2 x - \tau_1 x = -\iota_{ax} B - hx \label{eq:2symp2mor1}
\end{align}
Evidently, this equation uniquely determines $h$ from the rest of the data.

The $L_\infty$ homotopy equations read
\begin{align}
g_2x - g_1x &= -\tfrac{1}{2} a^* hx \label{eq:2symp2mor3} \\
g_2u - g_1u &= -\tfrac{1}{2} h(a'^*u) \label{eq:2symp2mor4} \\
\tg_2(x, y) - \tg_1(x, y) &= h[x, y] - [hx, g_1y] - [g_1x, hy]. \label{eq:2symp2mor5}
\end{align}

It is easy to see that equation \eqref{eq:2symp2mor4} follows from equations \eqref{eq:2symp1mor5} and \eqref{eq:2symp2mor1}. Similarly, equation \eqref{eq:2symp2mor5} follows from equations \eqref{eq:2symp1mor7} and \eqref{eq:2symp2mor1}.  We may now complete the proof of the main result:
\begin{proof}[Proof of \autoref{thm:shift2}]
Considering the computations at the level of objects, morphisms and two-morphisms, we have evidently produced a 2-functor
\[
\TCAConn(U)\rightarrow \Symp_2(U).
\]
By \autoref{lm:courantstrict} and \autoref{prop:symp-cour}, this 2-functor is essentially surjective, so we just have to show that it is fully faithful, i.e.~that for two twisted Courant algebroids $\cE_1,\cE_2$ and the corresponding two-shifted symplectic $L_\infty$ algebroids $\cL_1, \cL_2$, the functor
\[\Hom_{\TCAConn(U)}(\cE_1, \cE_2)\rightarrow \Hom_{\Symp_2(U)}(\cL_1, \cL_2)\]
is an equivalence of 1-groupoids.  It is clearly fully faithful since the 2-morphisms in both $\TCAConn(U)$ and $\Symp_2(U)$ are determined by a 2-form $B$ satisfying the same set of equations. To see that it is essentially surjective, we must show that any one-morphism in $\Hom_{\Symp_2(U)}(\cL_1,\cL_2)$ is equivalent to one in Courant form (i.e.~with $\tau = 0$).  But this follows immediately from \eqref{eq:2symp2mor1} and \eqref{eq:2symp2mor3}.
\end{proof}

\subsection{Classification of isotropic quotients}
\label{sec:shift2-iso}
Let $X$ be a manifold, and let $\SympIso_{2}(X)$ be $\infty$-groupoid of  two-shifted symplectic algebroids $(\cL,\omega)$ on $X$, equipped with an isotropic structure on the quotient map
\[
X \to \XL
\]
Then we immediately have the following result.
\begin{prop}\label{prop:shift2-iso}
For any manifold $X$, the $\infty$-groupoid $\SympIso_{2}(X)$ is equivalent to the 1-groupoid $\CAlgd(X)$ of Courant algebroids.  In particular, a twisted Courant algebroid is equivalent to an untwisted Courant algebroid if and only if its twisting class vanishes.
\end{prop}
\begin{proof}
It is enough to establish the claim for affine manifolds $U$.  The data of a two-shifted symplectic structure on an $L_\infty$ algebroid $\cL$ and an isotropic structure on $U \rightarrow \UL$ is equivalent to the data of a non-degenerate closed $(2,2)$-form in the homotopy fibre of the projection
\[
\Omcl^2(\UL) \to  \Omega^{\ge 2}(U).
\]
But by construction, this morphism is surjective; hence the homotopy fibre is equivalent to the strict fibre.  It follows that the $\infty$-groupoid $\SympIso_2(U)$ is equivalent to the subgroupoid of $\Symp_2(U)$ in which we set all differential forms that live purely on $U$ to zero.  This corresponds to setting $K=0$ on the level of objects, $H=0$ on the level of morphisms and $B=0$ on the level of 2-morphisms.  Via \autoref{thm:shift2}, this subgroupoid is naturally identified with the subgroupoid $\CAlgd(U) \subset \TCAlgd(U)$ of untwisted Courant algebroids.
\end{proof}

As a special case, recall that a Courant algebroid is \defterm{exact} if the anchor and its dual give an exact sequence of vector bundles
\[
\xymatrix{
0 \ar[r] & \cT^\vee_X \ar[r]^{a^*} & \cE \ar[r]^a & \cT_X \ar[r] & 0.
}
\]
Equivalently, the anchor $\cL \to \cT_X$ of the corresponding two-shifted symplectic algebroid is a quasi-isomorphism.

Thus exact Courant algebroids on $X$ are the same thing as two-shifted symplectic structures on $[X/\cT_X]$ together with an isotropic structure on the quotient $X \to [X/\cT_X]$.  But the tangent complex of $[X/\cT_X]$ is contractible, and hence the only symplectic structure is the trivial one.  Nevertheless, isotropic structures can be nontrivial: they are primitives for the zero element in $\Z^{2}(X,\Omega^{\ge 2}_X)$, i.e.~cocycles in $\Z^{1}(X,\Omega^{\ge 2}_X)$, and equivalences are provided by coboundaries.  In this way, we obtain a symplectic interpretation of  \v{S}evera's cohomological classification of exact Courant algebroids:

\begin{cor}[\cite{Bressler2005,Severa1998--2000,Severa2001}]\label{cor:exact-courant}
The stack of exact Courant algebroids is equivalent to the stack $\Omega^{\ge 2} [1]$ of 1-shifted closed two-forms.  Thus an exact Courant algebroid $\cE$ on a manifold $X$ is determined up to isomorphism by a class in $H^1(X, \Omega^{\ge 2}_X)$, called its ``\v{S}evera class'', and the group of base-fixing automorphisms of $\cE$ is the additive group of global closed two-forms on $X$.
\end{cor}

\subsection{Examples}
\label{sec:shift2-ex}
\subsubsection{The transitive case}
\label{sec:transitive}

A twisted Courant algebroid is \defterm{transitive} if its anchor map is surjective.  This is evidently equivalent to requiring that the $L_\infty$ algebroid $\cT_X \to \cE$ is a Lie 1-algebroid $\cE/\cT_X$.  Thus transitive twisted Courant algebroids are the same thing as classical Lie algebroids $\cL$ equipped with two-shifted symplectic structures as described in \autoref{sec:symplectic-transitive}.  From \autoref{prop:shift2-iso}, we immediately obtain the following result.
\begin{cor}[\cite{Bressler2007,Severa1998--2000}]\label{cor:pontryagin}
A quadratic Lie algebroid can be extended to a transitive Courant algebroid if and only if its first Pontryagin class vanishes.
\end{cor}
If $\g \subset \cL$ is the kernel of the anchor map and $U \subset X$ is an affine open subset, we can split $\cL|_U = \cT_U \oplus \g$.  The twisted Courant algebroid is then given locally by $\cE|_U = \cT_U \oplus \g \oplus \coT_U$, equipped with the obvious pairing. The bracket involves the Courant bracket on $\cT_U\oplus \coT_U$, the Lie bracket on $\g$ and the curvature of the splitting $\cT_U \to \cL$.  We refer the  reader to \cite{Bressler2007,Severa1998--2000} for the explicit formulae.

\subsubsection{Atiyah algebroids of perfect complexes}

Recall that a perfect complex on $X$ is a complex of quasi-coherent sheaves that is locally equivalent to a finite complex of finite rank vector bundles.   The classifying stack of perfect complexes carries a two-shifted symplectic structure~\cite{Pantev2013}.  So by analogy with the case $G$-bundles discussed in \autoref{sec:transitive}, we expect the Atiyah algebroid of a perfect complex $\cF = \cF^\bullet \in \Perf(X)$  to carry a two-shifted symplectic structure modeling the formal completion of $\Perf$ along the map $\cF : X \to \Perf$.  We refer the reader to \cite[Section 10.1]{Huybrechts1997} for an introduction to Atiyah algebroids of perfect complexes.

The Atiyah algebroid $\cL = \cL(\cF)$ sits in an exact triangle
\[
\xymatrix{
\mathbb{R}\cEnd(\cF) \ar[r] & \cL \ar[r] & \cT_X \ar[r] & \mathbb{R}\cEnd(\cF)[1],
}
\]
with the derived endomorphisms of $\cF$.  Thus $\cT_{[X/\cL]}$ is isomorphic to the complex on $\XL$ determined by the natural action of $\cL$ on $\mathbb{R}\cEnd(\cF)[1]$.  This complex carries the nondegenerate trace pairing, giving a two-shifted symplectic structure $\omega$ on $[X/\cA]$ with pullback
\[
[\pi^*\omega] = \mathrm{ch}_2(\cF) \in \coH^2(X,\Omega^{\ge 2}_X),
\]
the degree-two part of the Chern character.

But there is a subtlety: the Atiyah algebroid is not, in general, an $L_\infty$ algebroid: it will typically have cohomology in positive degrees, and thus be a \emph{derived} $L_\infty$ algebroid.  This corresponds geometrically to the fact that the stack $\Perf$ is quite singular.

So in order to obtain (underived) twisted Courant algebroids, we must rule out cohomology in positive degrees.  The long exact sequence in cohomology gives
\begin{align}
0 \to \cExt^{0}(\cF,\cF) \to \cH^0(\cL) \to \cT_X \to \cExt^{1}(\cF,\cF) \to \cH^1(\cL) \to 0 \label{eq:atiyah-ext}
\end{align}
and isomorphisms $\cH^i(\cL) \cong \cExt^{i}(\cF,\cF)$ for $i \ne 0,1$.   We remark that  $\cExt^1(\cF,\cF)$ is the sheaf of infinitesimal deformations of $\cF$ and the map $\cT_X \to \cExt^1(\cF,\cF)$ gives the infinitesimal deformations that arise by pulling back $\cF$ along flows of vector fields.  From the exact sequence, we obtain the following

\begin{prop}
Let $\cF$ be  a perfect complex on $X$.  Then the  Atiyah algebroid of $\cF$ is an $L_\infty$ algebroid if and only if 
\begin{align*}
\cExt^{i}(\cF,\cF) &= 0 \textrm{ for } i > 1
\end{align*}
and the natural map
\[
\cT_X \to \cExt^1(\cF,\cF),
\]
is surjective.  In this case, there is a canonical twisted Courant algebroid associated to $\cF$, whose twisting class is $\mathrm{ch}_2(\cF)$.
\end{prop}

Note that, while the Atiyah algebroid of a principal bundle is transitive, this is not the case for general perfect complexes.  Since the Atiyah algebroid represents the infinitesimal symmetries of the complex, its orbits are related to the stratification of $X$ by the singularities of $\cF$.

\subsubsection{Codimension-two cycles}

In the complex analytic or algebraic settings, the appearance of $H^2(X,\Omega^2_X)$ suggests a link between twisted Courant algebroids and codimension-two cycles.  Indeed, this is a special case of the previous example, as we now explain.

Suppose that $X$ is a complex manifold or smooth algebraic variety, and let $Y \subset X$ be a smooth subvariety of pure codimension two.  Let $\cI \subset \cO_X$ be the ideal sheaf of $Y$.   Then   $\cI$ is a coherent sheaf on $X$, and since $X$ is smooth, $\cI$ is perfect.  We recall the standard canonical isomorphisms
\begin{align}
\cExt^i(\cI,\cI) \cong \begin{cases}
\cO_X & i=0 \\
\cN_Y & i=1 \\
0 & \textrm{otherwise}
\end{cases}\label{eq:ext-normal}
\end{align}
where $\cN_Y$ is the normal bundle of $Y$, viewed as a coherent sheaf on $X$.  The connecting homomorphism
\[
\xymatrix{
\cT_X \ar[r] & \cExt^1(\cI,\cI) \cong \cN_Y
}
\]
is the natural projection of $\cT_X$ onto the normal bundle, which is surjective.  We remark that, although \eqref{eq:ext-normal} holds for an arbitrary local complete intersection, the surjectivity of the connecting homomorphism really requires $Y$ to be smooth.

We conclude that the Atiyah algebroid $\cL = \cL(\cI)$ is quasi-isomorphic to its zeroth cohomology, i.e.~$\cL$ is a coherent sheaf sitting purely in degree zero. But this sheaf is not a vector bundle; it has Tor amplitude $[-1,0]$, so that it is an honest Lie 2-algebroid.  From \eqref{eq:atiyah-ext}, we have an exact sequence
\[
\xymatrix{
0 \ar[r] & \cO_X \ar[r] & \cL \ar[r] & \cT_X(-\log Y) \ar[r] & 0
}
\]
where $\cT_X(-\log Y) \subset \cT_X$ is the kernel of the projection $\cT_X \to \cN_Y$, i.e.~the subsheaf of vector fields that are tangent to $Y$.  Thus the orbits of $\cL$ are the connected components of $Y$ and its complement. Since  $\mathrm{ch}_2(\cI)$ is, up to sign, the class $[Y]$ Poincar\'e dual to $Y$, we arrive at the
\begin{thm}
Let $Y \subset X$ be a smooth codimension-two subvariety in a complex manifold or smooth algebraic variety.  Then there is a canonical twisted Courant algebroid on $X$ whose twisting class is $-[Y] \in \coH^2(X,\Omega^{\ge 2}_X)$, and whose orbits are the connected components of $Y$ and its complement.
\end{thm}

This twisted Courant algebroid is always locally equivalent to a Courant algebroid, which we can describe concretely as follows.  On a sufficiently small affine open subset $U \subset X$, we can find a flat rank-two vector bundle $(\cV,\nabla)$ and a section $s \in H^0(U,\cV)$ whose zero scheme is $Y \cap U$.  In this way we obtain the Koszul resolution
\[
\cI|_U \cong (\xymatrix{ \det \cV^\vee \ar[r]^-{s}& \cV^\vee }),
\]
which we may use to compute the derived endomorphisms $\mathbb{R}\cEnd(\cI|_U)$ and the Atiyah algebroid $\cL|_U$.

In degree zero, we get $\g = \mathbb{R}\cEnd^0(\cI|_U)= \cEnd(\cV) \oplus \cO_U$ with the obvious Lie bracket, and with pairing given by the difference of the trace pairings.  The connection on $\cV$ allows us to identify $\cL|_U = \cT_U\oplus \g$ so that we get a transitive Courant algebroid 
\[
\cE_0 = \cT_U \oplus \g \oplus \cT^\vee_U,
\]
as in \autoref{sec:transitive}.

Identifying the degree-one piece of $\cL|_U$ with the bundle $\RR\cEnd^1(\cI|_U) = \cV$, the differential on $\cL|_U$ is given by the map
\[
\mapdef{\delta}{ \cT_U\oplus \g}{ \cV}
{ (\xi,(\phi,f))  }{\nabla_\xi s+ \phi s - fs }
\]
for $\xi \in \cT_U$, $(\phi,f) \in \g$.  Because $s$ vanishes transversely, this map is surjective. 

Now $\delta$ evidently extends to a surjection $\cE_0 \to \cV$ whose kernel $\cK \subset \cE_0$ is a coisotropic subbundle that is preserved by the Courant bracket.  The annihilator $\cK^\perp \subset \cK$ is the image of $\cV^\vee$ under the dual map $\cV^\vee \to \cE^\vee \cong \cE$.  In this way, we obtain the desired Courant algebroid by coisotropic reduction:
\begin{align*}
\cE = \cK/\cK^\perp = \cH^0(\cV^\vee \to \cE_0 \to \cV)
\end{align*}
which its induced bracket, anchor and pairing.

\section{Two-shifted Lagrangians}
\label{sec:lagrangians}

Let $\cL$ be a two-shifted symplectic algebroid on $X$, and let $\cE$ be the corresponding twisted Courant algebroid.   Although the projection map $X\rightarrow [X/\cL]$ may be isotropic as above, it is essentially never Lagrangian.  Indeed, we recall from \autoref{ex:lag-quotient} that the Lagrangian condition forces  $\cL \cong \cT^\vee_X[1]$, which means that $\cE = 0$.  Nevertheless, there may be many Lagrangians of the form
\[
[Y/\cM]\rightarrow [X/\cL]
\]
where $Y \subset X$ is a closed submanifold, and $\cM$ is an $L_\infty$ algebroid on $Y$.  In this section, we give a classification of such Lagrangians in terms of twisted Dirac structures in twisted Courant algebroids. 

\subsection{Twisted Dirac structures}

Let $X$ be a manifold, and let $f \colon Y \to X$ be the inclusion of a closed submanifold.  Suppose that $\cE$ is a twisted Courant algebroid on $X$ whose twisting cocycle lies in the relative de Rham complex $\Omega^\bullet_{X,Y} = \ker(\Omega^\bullet_X \to \Omega^\bullet_Y)$.

The restriction $f^*\cE$ is a twisted vector bundle on $Y$ equipped with a nondegenerate symmetric pairing, and so it makes sense to speak of twisted subbundles $\cF \subset f^*\cE$ that are isotropic or Lagrangian. Here, by a ``twisted subbundle'', we mean that on any affine chart, $\cF$ is a subbundle of $\cE$, and on the overlap of two charts, the subbundles are preserved by the transition functions of $\cE$.

 Applying the anchor to such a twisted subbundle, one obtains a subsheaf $a(\cF) \subset f^*\cT_X$.  We say that $\cF$ is \defterm{compatible with the anchor} if $a(\cF)\subset \cT_Y$.  In this case, it is easy to see that on any affine chart $U \subset X$ there is a well-defined bracket
\[
\cour{-,-}\colon \cF \times \cF \to f^*\cE
\]
defined by restriction of the Courant bracket on $\cE$.  In particular, it makes sense to ask if $\cF$ is \defterm{involutive}, i.e.~$\cour{\cF,\cF}\subset \cF$.  This allows us to extend the definition of a Dirac structure with support~\cite{Alekseev,Bursztyn2009,Severa2005} to the twisted setting:

\begin{defn}
Let $f \colon Y \to X$ be an embedding of a closed submanifold.  A \defterm{twisted Dirac pair on $(X,Y)$} is a pair $(\cE,\cF)$ consisting of a twisted Courant algebroid $\cE$ whose twisting cocycle lies in $\Omega^{\ge 2}_{X,Y} \subset \Omega^{\ge 2}_X$, and a twisted Lagrangian subbundle $\cF \subset f^*\cE$  that is compatible with the anchor and involutive.
\end{defn}

Twisted Dirac pairs are the objects of a natural 2-groupoid $\TDirac(X,Y)$.  For pairs $(\cE,\cF)$ and $(\cE',\cF')$ the morphisms are given by the subgroupoid
\[
\Hom_{\TDirac(X,Y)}((\cE,\cF),(\cE',\cF')) \subset \Hom_{\TCAlgd(X)}(\cE,\cE')
\]
consisting of the morphisms in $\TCAlgd$ that preserve the twisted subbundles, and for which all form data lie in $\Omega^{\ge 2}_{X,Y}\subset \Omega^{\ge 2}_X$.
\begin{remark}
The natural forgetful map $\TDirac(Y, X)\rightarrow \TCAlgd(X)$ is not a fibration, so to define the space of Dirac structures in a fixed twisted Courant algebroid $\cE \in \TCAlgd(X)$, we must take its homotopy fibre in $\TDirac(X,Y)$ instead of its strict fibre.  For example, we should choose an isomorphism of $\cE$ with an equivalent model $\cE'$ for which the twisting cocycle lies in $\Omega^2_{X,Y}\subset \Omega^2_X$. \qed
\end{remark}

If $(\cE,\cF)$ is a Dirac pair on $(X,Y)$, the anchor gives a morphism
\[a^* \colon \cN^\vee_Y\rightarrow \cF.\]
Indeed, consider the diagram
\[
\xymatrix{
0 \ar[r] & \cN^\vee_Y \ar[r] \ar@{-->}[d] & f^* \cT^\vee_X \ar[r] \ar^{a^*}[d] & \cT^\vee_Y \ar[r] \ar^{a^\vee}[d] & 0 \\
0 \ar[r] & \cF \ar[r] & f^*\cE\ar[r] & \cF^\vee \ar[r] & 0
}
\]
The top sequence is exact by definition. The bottom sequence is exact since $\cF\subset f^*\cE$ is Lagrangian. Therefore, we get a unique morphism $\cN^\vee_Y\rightarrow \cF$ denoted by the dashed arrow.

The twisting class of the pair $(\cE,\cF)$ is evidently a refinement of the twisting class $[\cE] \in \coH^2(\Omega^{\ge 2}_X)$ to a class in relative cohomology:
\[
[\cE,\cF] \in \coH^2(\Omega^{\ge 2}_{X,Y}).
\]
Notice that the exterior power of the exact sequence
\[
\xymatrix{
0 \ar[r] & \cN^\vee_Y \ar[r] & f^*\Omega^1_X \ar[r] & \Omega^1_Y \ar[r] & 0
}
\]
gives rise to a natural projection $\Omega^2_{X,Y} \to\Omega^1_Y \otimes \cN^\vee_Y$. The image of $[\cE,\cF]$ under the resulting map
\[
\coH^2(\Omega^{\ge 2}_{X,Y}) \to \coH^2(Y,\Omega^1_Y\otimes \cN^\vee_Y)
\]
measures the twisting of the transition functions of $\cF$, as is evident from the formula~\eqref{eqn:bundle-twist} for the twisting of $\cE$.

\subsection{Classification of two-shifted Lagrangians}

We now prove the following classification of two-shifted Lagrangians:

\begin{thm}
Let $Y \subset X$ be a closed submanifold.  Then the $\infty$-groupoid $\Lag_2(X,Y)$ parametrizing two-shifted symplectic  $L_\infty$ algebroids $\cL$ on $X$ together with a  Lagrangian $[Y/\cM]\rightarrow [X/\cL]$ is equivalent to the 2-groupoid $\TDirac(X,Y)$ of twisted Dirac pairs.
\label{thm:shift2lagrangian}
\end{thm}

The strategy of the proof is parallel to that of \autoref{thm:shift2}. Once again, we fix an affine manifold $U$, this time with a closed submanifold $V \subset U$, and we consider a diagram of 2-groupoids
\[
\xymatrix{
& \TDiracconn(U,V) \ar[dl] \ar^{\sim}[dr] & \\
\Lag_2(U,V) && \TDirac(U,V)
}
\]
where the objects of $\TDiracconn(U,V)$ are twisted Dirac pairs in which the twisted Courant algebroids is equipped with a metric connection and the Dirac structures is equipped with a tensor $\nu\in\wedge^2\cF^\vee\otimes \cN^\vee_V$. We will construct a functor $\TDiracconn(U,V)\rightarrow \Lag_2(U,V)$ which we will prove is an equivalence.

\subsubsection{Objects}

We begin by constructing the equivalence on the level of objects.  Using \autoref{thm:shift2} we identify
\[\cL\cong \rbrac{\xymatrix{\cT^\vee_U \ar[r]& \cE}}\]
for a twisted Courant algebroid $\cE$ on $U$, with symplectic form $\omega \in \Omcl^{2,2}(\UL)$.  It is then easy to see that a Lagrangian structure forces $\cM$ to be concentrated in degrees $-1$ and $0$.

An $L_\infty$-morphism $\cM \to \cL$ is given by morphisms
\begin{align*}
g &\colon \cM\rightarrow f^*\cL & \tg &\colon \wedge^2 \cM_0\rightarrow f^*\cT^\vee_U
\end{align*}
compatible with the anchor and satisfying \eqref{eq:2symp1mor1}--\eqref{eq:2symp1mor3}.

An isotropic structure on $[V/\cM]\rightarrow [U/\cL]$ is given by elements
\begin{align*}
\tau &\in \cM_0^\vee\otimes \Omega^1(V) & H&\in \Omega^3(V)
\end{align*}
satisfying $(f,g)^*\omega = \dtw(\tau + H)$, which gives the following equations analogous to \eqref{eq:2symp1mor4}--\eqref{eq:2symp1mor7}:
\begin{align}
f^* K &= \ddr H \label{eq:diracobj1} \\
f^* gu &= -\tau \delta u \label{eq:diracobj2} \\
\abrac{ gx, gy } &= \tfrac{1}{2}(\iota_{ax} \tau y + \iota_{ay}\tau x) \label{eq:diracobj3} \\
f^*\psi(gx, gy) - f^*\tg(x, y) &= \tau[x, y] - \lie{ax} \tau y + \lie{ay}\tau x \nonumber \\
&\ \ \ + \tfrac{1}{2}\ddr(\iota_{a x }\tau y - \iota_{a y} \tau x) + \iota_{a x}\iota_{ay} H. \label{eq:diracobj4}
\end{align}
 for $x,y \in \cM_0$ and $u \in \cM_1$.  We will now show that one can rectify the Lagrangian $[V/\cM]\rightarrow [U/\cL]$.

\begin{prop}
Let $\cL$ be a two-shifted symplectic $L_\infty$ algebroid on $X$ corresponding to a twisted Courant algebroid $\cE$ and let $[Y/\cM]\rightarrow [X/\cL]$ a Lagrangian. Then $\cM$ is quasi-isomorphic to a subcomplex
\[
\cM \cong \rbrac{\xymatrix{\cN_V^\vee \ar[r] & \cF}} \subset \rbrac{\xymatrix{f^*\coT_U \ar[r] & f^*\cE}}
\]
where $\cF \subset f^*\cE$ is a Lagrangian subbundle, and the symplectic structure vanishes identically on $\cM$.
\label{lm:2lagstrict}
\end{prop}
\begin{proof}
We consider the complex $\widetilde{\cM} = \cM\oplus f^*\coT_U\oplus f^*\cT^\vee_U[1]$ with differential
\[
\xymatrix{
f^*\coT_U\oplus \cM_1 \ar[r]^{\begin{psmallmatrix}\id & g \\ 0 & \d\end{psmallmatrix}} & f^*\coT_U\oplus \cM_0
}
\]
and the complex $\widetilde{\cF}=\cM\oplus f^*\cT^\vee_U$ with the differential twisted by $g$.

As in the proof of \autoref{lm:courantstrict} we get a deformation retract of the form $p\colon \widetilde{\cM}\rightleftarrows \cM\colon i$, so we can transfer the $L_\infty$ algebroid structure from $\cM$ to $\widetilde{\cM}$. The Lagrangian structure on $\widetilde{\cM}$ takes the following shape:
\[
\xymatrix{
\pi^*\cT_{[V/\widetilde{\cM}]} \ar[d] &  f^*\cT^\vee_U \oplus \cM_1 \ar[r] \ar^{(0, g)}[d] & f^*\cT^\vee_U \oplus \cM_0 \ar[r] \ar^{(0, g)}[d] \ar@{=>}[ldd] & \cT_V \ar[d]\ar@{=>}[ldd]\\
f^*\pi^*\cT_\UL\ar[d]^{\omega} & f^* \cT^\vee_U \ar[r] \ar[d] & f^*\cE \ar[r] \ar[d] & f^*  \cT_U \ar[d] \\
\pi^*\cT_{[V/\widetilde{\cM}]}[2] & \cT^\vee_V \ar[r] & f^*\cT_U \oplus \cM_0^\vee \ar[r] & f^*\cT_U \oplus \cM_1^\vee
}
\]
with the null homotopy of the composite given by $\tau\in \cM_0^\vee\otimes \cT^\vee_V$.

 The identity operator on $\coT_U$ gives a projection $\widetilde{\cM}_0\rightarrow f^*\coT_U$, which we use as a homotopy to modify the map $\widetilde{\cM}\rightarrow \cM\rightarrow f^*\cL$. Using the formulae in \autoref{sec:2shift-2mor}, we see that the modified Lagrangian structure has the form
\[
\xymatrix{
f^*\cT^\vee_U \oplus \cM_1 \ar[r] \ar^{(\id, 0)}[d]& f^*\coT_U \oplus \cM_0 \ar[r] \ar^{(a^\vee, g)}[d] \ar@{=>}[ldd] & \coT_V \ar[d] \ar@{=>}[ldd] \\
f^* \coT_U \ar[r] \ar[d] & f^*\cE \ar[r] \ar[d] & f^*\coT_U \ar[d] \\
\coT_V \ar[r] & f^*\cT_U \oplus \cM_0^\vee \ar[r] & f^*\cT_U \oplus \cM_1^\vee
}
\]
with the null homotopy now given by the inclusion $\cT_V\subset f^*\cT_U$ and its dual.

The nondegeneracy condition on the Lagrangian $[V/\cM]\rightarrow [U/\cL]$ now implies that we have a self-dual exact triangle
\[
\xymatrix{
\widetilde{\cF} \ar[r] & \cT_V\oplus f^*\cE\oplus \cT^\vee_V \ar[r] & \widetilde{\cF}^\vee \ar[r] & \widetilde{\cF}[1]
}
\]
Since $\cT_V\oplus f^*\cE\oplus \coT_V$ is concentrated in degree zero, and $\widetilde{\cF}$ is concentrated in nonpositive degrees, we conclude that the projection $\widetilde{\cF}\rightarrow \cH^0(\widetilde{\cF})$ is a quasi-isomorphism. This allows us to replace the $L_\infty$ algebroid $\widetilde{\cM}$ by $f^*\cT_U^\vee\rightarrow \cH^0(\widetilde{\cF})$, giving a Lagrangian structure of the form
\[
\xymatrix{
f^*\cT^\vee_U \ar[r] \ar^{\id}[d] & \cH^0(\widetilde{\cF})\ar@{=>}[ldd] \ar[r] \ar[d] & \cT_V \ar[d] \ar@{=>}[ldd] \\
f^* \cT^\vee_U \ar[r] \ar[d] & f^*\cE \ar[r] \ar[d] & f^*\cT_U \ar^{\id}[d] \\
\cT^\vee_V \ar[r] & \cH^0(\tilde{\cF})^\vee \ar[r] & f^*\cT_U
}
\]
with homotopy $\tau \in \cH^0(\widetilde{\cF})^\vee \otimes \coT_V$.  From equation \eqref{eq:diracobj2} we see that $\tau$ is surjective, and hence we may define a bundle $\cF$ by the commutative diagram
\[
\xymatrix{
0 \ar[r] & \cN^\vee_V \ar[r] \ar@{-->}[d] & f^*\cT^\vee_U \ar[r] \ar[d] & \cT^\vee_V \ar[r] \ar^{\id}[d] & 0 \\
0 \ar[r] & \cF \ar[r] & \cH^0(\widetilde{\cF}) \ar^{\tau}[r] & \cT^\vee_V \ar[r] & 0
}
\]
with exact rows.  Considering the columns as morphisms of two-term complexes, we obtain the desired quasi-isomorphism $\cM \cong \rbrac{\cN^\vee_V \to \cF}$.
\end{proof}

We now assume that $\cM\cong (\cN^\vee_V\rightarrow \cF)$ as in the Proposition, so that the morphism $g\colon \cM \to f^*\cL$ is the inclusion.  The space of pairs of a closed form on $\UL$ and an isotropic structure on $\VM \to \UL$ is given by the the homotopy fibre of the projection $\Omega^\bullet(\UL)\rightarrow \Omega^\bullet(\VM)$.  Since the projection is now surjective, this is just the kernel.   Thus we may assume that all the form data on $\VM$, namely $\tau$ and $H$, are zero.

The equations  \eqref{eq:2symp1mor1}, \eqref{eq:2symp1mor2} and \eqref{eq:2symp1mor3} for an $L_\infty$ morphism now uniquely determine, the binary bracket $\cM_0\times \cM_0\rightarrow \cM_0$, the binary bracket $\cM_0\times \cM_1\rightarrow \cM_1$, and  the triple bracket $\wedge^3 \cM_0 \rightarrow \cM_1$, respectively. The isotropy condition \eqref{eq:diracobj4} then implies that
\[
\tg (x, y) = \psi(x, y) + \nu(x, y)
\]
for some $\nu\in \wedge^2\cF^\vee\otimes f^*\cN^\vee_V$.  Hence by the defining relation \eqref{eq:2sympcourant} between the Courant bracket and the binary bracket on $\cL$, the equations for the isotropic structure reduce to the single condition
\begin{align*}
[x, y]_\cM &= [x,y]_\cL - \tfrac{1}{2} a_{\cE}^* \tilde{g}(x, y) \\
&= \cour{x,y} - \tfrac{1}{2}a_{\cE}^* \ddr \abrac{x,y}- \tfrac{1}{2}a_{\cF}^* \nu(x, y),
\end{align*}
where  the expressions $\cour{x,y}$ and $\ddr \abrac{x,y}$ are defined by extending $x$ and $y$ to sections of $\cE$ in a neighbourhood of $V$ and then restricting the results.  Since $\cF\subset f^*\cE$ is isotropic, the expression $\ddr \abrac{x,y}$ automatically lies in $\cN_V$, which implies that $\cour{x, y}\in \cF$, so that $\cF$ is a Dirac structure in $\cE$.  This gives the equivalence of $\TDiracconn(U,V)$ and $\Lag_2(U,V)$ at the level of objects.

\subsubsection{1-morphisms}
The one-morphisms in $\Lag_2(U,V)$ are given by homotopy commutative diagrams of the form
\[
\xymatrix{
[V/\cM] \ar[d] \ar^{\mu}[r] & [V/\cM'] \ar[d] \ar@{=>}^{h}[dl]\\
[U/\cL] \ar_{g}[r] & [U/\cL']
}
\]
where the vertical morphisms are two-shifted Lagrangians defined by twisted Dirac pairs.

An $L_\infty$ morphism $\cM \rightarrow \cM'$ consists of a chain morphism $\mu\colon \cM \rightarrow \cM'$ and a linear map $\widetilde{\mu}\colon \wedge^2 \cF \rightarrow \cN^\vee_Y$ satisfying \eqref{eq:2symp1mor1}--\eqref{eq:2symp1mor3} with $g$ replaced by $\mu$.  Meanwhile by \autoref{thm:shift2}, the 1-morphism $\UL \to [U/\cL']$ is determined by a morphism of twisted Courant algebroids, consisting of a bundle map $g  \colon \cE \rightarrow \cE'$ and a three-form $H\in \Omega^3(U)$.  Finally, we have the data of a homotopy between the composites
\[
\rbrac{\xymatrix{
\cM \ar[r]^{\mu} & \cM' \ar[r]& f^*\cL'
}} \sim_h \rbrac{{\xymatrix{\cM \ar[r] & f^*\cL \ar[r]^{g} & f^*\cL'}}},
\]
compatible with the form data.   It is determined by a bundle map $h\colon \cF\rightarrow f^*\cT^\vee_X$ that satisfies the following variants of \eqref{eq:2symp2mor3}--\eqref{eq:2symp2mor5}:
\begin{align}
\mu x - gx &= -\tfrac{1}{2} a^\vee h \label{eq:dirac1mor1} \\
\mu u - u &= -\tfrac{1}{2} h a^\vee u \label{eq:dirac1mor2} \\
h[x, y] - [hx,y] - [x,hy] &= \psi'(\mu x,\mu y) -\nu'(\mu x, \mu y) \nonumber \\& \ \ \ + \widetilde{\mu}(x, y) - \tg(x,y) \nonumber \\
&\ \ \ - \psi(x,y) + \nu(x, y) \label{eq:dirac1mor3}
\end{align}

Since we assume the Lagrangian structure to be strict, we get $f^* H = 0$ and $f^* h x = 0$. In particular, $h\in \cF^\vee\otimes \cN^\vee_V$.  

Consider the case when $h=0$. Then equation \eqref{eq:dirac1mor1} holds if and only if $g$ intertwines the subbundle $\cF \subset f^*\cE$ and $\cF'\subset f^*\cE'$, and equation \eqref{eq:dirac1mor2} means  that $\mu$ in degree $-1$ is the identity. Equation \eqref{eq:dirac1mor3} uniquely determines $\tilde{\mu}(x, y)$.  It is not difficult to check that then the equations \eqref{eq:2symp1mor1}--\eqref{eq:2symp1mor3} for $\mu$ to be an $L_\infty$ morphism are automatically satisfied.  Thus we see that there is an inclusion $\TDiracconn(U,V) \to \Lag_2(U,V)$ at the level of one-morphisms.

\subsubsection{2-morphisms}

2-morphisms in $\Lag_2(U,V)$ are given by diagrams of the form
\[
\xymatrix@R+20pt{
[V/\cM] \ar[d] \rtwocell^{\mu_1}_{\mu_2}{\chi} & [V/\cM'] \ar[d] \ar@{=>}^{h}[dl] \\
[U/\cL] \rtwocell^{g_1}_{g_2} & [U/\cL']
}
\]
together with a homotopy $\mu_1\stackrel{\chi}\sim \mu_2$ represented by morphism $\chi\colon \cF \rightarrow \cN^\vee_V$, and  a 2-morphism $(g_1,H_1) \sim (g_2,H_2)$ represented by a two-form $B \in \Omega^{2}(U)$ such that $f^*B = 0$.

The homotopies satisfies the following equations:
\begin{align}
\mu_2 x - \mu_1 x &= -\tfrac{1}{2} a^\vee \chi x \label{eq:dirac2mor1} \\
\mu_2 u - \mu_1 u  &= -\tfrac{1}{2} \chi a^\vee x \label{eq:dirac2mor2} \\
\widetilde{\mu}_2(x, y) - \widetilde{\mu}_1(x, y) &= \chi [x, y] - [\chi x, \mu_1y] - [\mu_1x, \chi y] \label{eq:dirac2mor3} \\
h_2x - h_1 x &= \chi x  \label{eq:dirac2mor4}
\end{align}

Equation \eqref{eq:dirac2mor2} is automatic in view of equations \eqref{eq:dirac2mor4} and \eqref{eq:dirac1mor2}. Similarly, equation \eqref{eq:dirac2mor3} follows from equations \eqref{eq:dirac2mor4} and \eqref{eq:dirac1mor3}.  Thus \eqref{eq:dirac2mor4} implies that every 1-morphism is equivalent to one for which $h=0$.  Restricting to such 1-morphisms, we see that $\chi=0$, and hence $\mu_2=\mu_1$.

In this way we have produced a 2-functor $\TDiracconn(U,V)\rightarrow \Symp_2(U,V)$ which is clearly fully faithful and essentially surjective by \autoref{lm:2lagstrict}.

\subsection{Untwisted Dirac structures}
\label{sec:shift2-corr}
The more classical  notion of a Dirac structure concerns Lagrangian subbundles in an untwisted Courant algebroid. More specifically, we have a 1-groupoid $\Dirac(X,Y)$ as follows:
\begin{itemize}
\item[] \textbf{Objects} of $\Dirac(X, Y)$ are given by Dirac pairs $(\cE,\cF)$ for which $\cE$ is an (untwisted) Courant algebroid.

\item[] \textbf{1-morphisms} $(\cE, \cF)\rightarrow (\cE', \cF')$ in $\Dirac(X,Y)$ are given by Courant algebroid isomorphisms $g\colon \cE\rightarrow \cE'$  that preserve the Dirac structures.
\end{itemize}

The symplectic counterpart is as follows.  Consider a two-shifted isotropic structure $X \to \XL$ representing an untwisted Courant algebroid $\cE$. Then a Dirac structure in $\cE$ is the data of a Lagrangian $\YM \to \XL$, forming a commutative square
\[
\xymatrix{
Y \ar[r] \ar[d]& X \ar[d] \\
\YM \ar[r] & \XL,
}
\]
together with a homotopy of the isotropic structures on $Y \to \XL$ induced by the two compositions. Such a diagram is an example of an \defterm{isotropic correspondence}.  We denote by $\LagIso_2(Y, X)$ the $\infty$-groupoid parametrizing such structures and their homotopies.  An argument identical to the proof of \autoref{prop:shift2-iso} then gives the
\begin{prop}
Let $Y \subset X$ be a closed submanifold.  Then the $\infty$-groupoid $\LagIso_2(Y, X)$ is equivalent to the 1-groupoid $\Dirac(X,Y)$ of untwisted Dirac pairs.
\end{prop}

Consider now the special case in which $\cL = \cT_X$ and $\cM = \cT_Y$, so that we have an isotropic correspondence 
\[
\xymatrix{
Y \ar[d]\ar[r] & X \ar[d] \\
\YT \ar[r] & \XT
}
\]
On the one hand, all form data on $\YT$ and $\XT$ are equivalent to zero, so that the isotropic structure on $X$ is a 1-cocycle in $\Omega^{\ge 2}_X$, and the isotropic correspondence is given by a trivialization of its pullback to $Y$.

On the other hand, the isotropic structure $X \to \XT$ corresponds to an exact Courant algebroid $\cE$ on $X$ as in \autoref{sec:shift2-iso}.  It pulls back to an isotropic structure $Y\to \YT$, giving an exact Courant algebroid $f^!\cE$ by the usual restriction  formula
\[
f^!\cE = a^{-1}(\cT_Y) / \cN^\vee_Y.
\]
The isotropic correspondence then results in a trivialization 
\[
f^!\cE \cong \cT_Y \oplus \coT_Y
\]
Under this isomorphism, the Lagrangian $\YT \to \XT$ corresponds to the Dirac structure $\cF \subset f^*\cE$ given by the preimage of $\cT_Y$ along the projection $a^{-1}(\cT_Y) \to f^!\cE$.  It sits in an exact sequence
\begin{align*}
\xymatrix{
0 \ar[r] & \cN^\vee_Y \ar[r] & \cF \ar[r] & \cT_Y \ar[r] & 0
}
\end{align*}
and is referred to in \cite[Section 6]{Gualtieri2011} as the \defterm{generalized tangent bundle} of the Courant trivialization.  We arrive at the following Lagrangian analogue of \autoref{cor:exact-courant}:

\begin{cor}\label{cor:gen-tan}
Let $\cE$ be an exact Courant algebroid on $X$.  The set of generalized tangent bundles for $Y \subset X$ is in bijection with the set of Courant trivializations of $f^!\cE$.  These sets are torsors for the space $\coH^0(Y,\Omcl^2)$ of global closed two-forms on $Y$. 
\end{cor}

\section{One-shifted symplectic forms}
\label{sec:shift1}

\subsection{Symplectic structures and exact Dirac pairs}
\label{sec:shift1-symp}
The classification of one-shifted symplectic algebroids is a simple consequence of the results so far.  Indeed, we recall that a one-shifted symplectic structure on $\XL$ is the same thing as a Lagrangian structure on
\[
\XL \to \XT,
\]
where $\XT$ carries the zero two-shifted symplectic structure.  To see why,  observe that an isotropic structure is now, by definition, a primitive for the trivial closed $(2,2)$-form on $\XL$, i.e.~a closed $(2,1)$-form.  This isotropic structure is Lagrangian if and only if the induced sequence
\[
\xymatrix{
\cT_\XL \ar[r] & \cT_\XT \ar[r] & \cT^\vee_\XL[2] \ar[r] & \cT_\XL[1]
}
\]
is an exact triangle.  Since the tangent complex $\cT_\XT$ is contractible, this is equivalent to requiring that the 1-shifted two-form on $\XL$ gives a quasi-isomorphism  $\cT_\XL \cong \cT^\vee_\XL[1]$, i.e.~that it is symplectic. 

Now a Lagrangian structure on $\XL \to \XT$ automatically induces an isotropic structure on $X \to \XT$ by pullback.  By \autoref{cor:exact-courant}, this gives an exact Courant algebroid $\cE$ on $X$, and the Lagrangian structure embeds $\cL$ as a Dirac structure in $\cE$.  We call such pairs $(\cE,\cL)$ \defterm{exact Dirac pairs}; they form a full subgroupoid $\ExDirac(X) \subset \Dirac(X,X)$. We therefore recover the infinitesimal characterization~\cite{Bursztyn2004,Xu2004} of quasi-symplectic groupoids:
\begin{thm}\label{thm:shift1}
For any manifold $X$, there is a canonical equivalence,
\[
\Symp_1(X) \cong \ExDirac(X)
\]
between the $\infty$-groupoid of 1-shifted symplectic algebroids on $X$, and the 1-groupoid of exact Dirac pairs.
\end{thm}

\begin{remark}\label{rmk:affine-exact-dirac}
On an affine manifold $U$, we can choose a splitting of the exact sequence
\[
\xymatrix{
0 \ar[r] & \cT_U^\vee \ar[r] & \cE \ar[r] & \cT_U \ar[r] & 0,
}
\]
giving an identification of $\cE$ with the standard Courant algebroid $ \cT_U \oplus \cT_U^\vee$, but with the bracket twisted by a closed three-form $H \in \Omcl^3(U)$ as  in \autoref{ex:h-twist-std}.   In this presentation, the $(2,1)$-form underlying the shifted symplectic structure on $\XL$ is completely determined by the following quasi-isomorphism of complexes on $X$, associated with the embedding of $\cL$ as a Lagrangian subbundle in $\cE \cong \cT_U \oplus \cT_U^\vee$:
\[
\xymatrix{\pi^*\cT_\UL \ar[d]& 
\cL \ar[r]\ar[d] &  \cT_U \ar[d] \\
\pi^*\cT^\vee_\UL[1] & \cT^\vee_U\ar[r] & \cL^\vee
}
\]
Meanwhile, the closure data is provided by the three-form $H$ that modifies the standard Courant bracket.  In the non-affine setting, $\cE$ need not split, so such a description of the symplectic structure may only exist locally. \qed
\end{remark}
There is also a symplectic interpretation of the ``tensor product'' of suitably transverse exact Dirac structures,  introduced in \cite{Alekseev2009,Gualtieri2011}.  Indeed, let $(\cE_1,\cL_1)$ and $(\cE_2,\cL_2)$ be exact Dirac pairs.  Then we have a pair of two-shifted Lagrangian morphisms $[X/\cL_i] \to [X/\cT_X]$.  Let us denote by $\overline{[X/\cL_2]}$ the Lagrangian in which the signs of all form data are reversed.  Then we can define a new Lie algebroid $\cL_1 \boxtimes \cL_2$ by taking the fibre product
\[
 [X/\cL_1] \underset{[X/\cT_X]}{\times} \overline{[X/\cL_2]} \cong [X/(\cL_1 \boxtimes \cL_2)].
 \]
Since the fibre product of two $n$-shifted Lagrangians is always $(n-1)$-shifted symplectic~\cite{Pantev2013}, we see that $[X/(\cL_1 \boxtimes \cL_2)]$ carries a canonical one-shifted symplectic structure, and hence we obtain a new exact Dirac pair $(\cE_1\boxtimes \cE_2,\cL_1\boxtimes \cL_2)$.

To see that this one-shifted symplectic structure coincides with the construction in \cite{Alekseev2009,Gualtieri2011}, we note that the underlying isotropic structures are additive under $\boxtimes$ and hence the \v{S}evera class of the exact Courant algebroid is additive under this operation, which determined $\cE_1\boxtimes \cE_2$ up to isomorphism; see  \cite[p.~88]{Gualtieri2011} for the functorial construction.  Considering the corresponding fibre product on tangent complexes, we immediately see that 
\[
\cL_1 \boxtimes \cL_2 \cong \cL_1 \underset{\cT_X}{\times}\cL_2 \cong \rbrac{\xymatrix{
\cL_1 \oplus \cL_2 \ar[r] & \cT_X
}}
\]
where $\cL_1 \oplus \cL_2$ sits in degree zero.  When the anchor maps are transverse, this complex is quasi-isomorphic to its zeroth cohomology, which is the usual fibre product of vector bundles, giving the desired formula for the Dirac structure.   When the anchors are not transverse, one could still make sense of the tensor product as some derived version of a Dirac structure, which would have nontrivial cohomology in degree one.

To see that $\boxtimes$ is monoidal, notice that the associativity of Lagrangian fibre products implies the associativity relation $(\cL_1 \boxtimes \cL_2) \boxtimes \cL_3 \cong \cL_1 \boxtimes(\cL_1 \boxtimes \cL_2)$
on exact Dirac structures.  The  monoidal unit is given by the identity Lagrangian $\XT \to \XT$, which  corresponds to the Dirac structure $\cT_X \subset \cT_X \oplus \cT^\vee_X$.

\subsection{Lagrangians with support and Courant trivializations}
\label{sec:shift1-lagrangian}

Now suppose that $X$ is a manifold, and $(\cE,\cL)$ is an exact Dirac pair, defining a one-shifted symplectic structure on $\XL$.  We now classify Lagrangians of the form $\YM \to \XL$, where $f \colon Y \to X$ is a closed submanifold.

We view the symplectic structure on $\XL$ as a Lagrangian $\XL \to \XT$ as above, and consider the commutative diagram
\[
\xymatrix{
Y \ar[d]\ar[r] & X \ar[d] \\
\YM \ar[r]\ar[d] & \XL \ar[d]\\
\YT \ar[r] & \XT
}
\]
The outermost rectangle is then an isotropic correspondence of the type considered in \autoref{sec:shift2-corr}, and so we obtain a Courant trivialization on $Y$.  Let $\cF \subset f^*\cE$ be the corresponding generalized tangent bundle.  Then the tangent complexes of the bottom square give the diagram
\[
\xymatrix{
(\cM\to \cT_Y) \ar[r] \ar[d] & (f^*\cL \to f^*\cT_X) \ar[d] \\
(\cN^\vee_Y \to \cF \to \cT_Y) \ar[r] & (f^*\coT_X \to  f^*\cE \to f^*\cT_X)
}
\]
and we immediately see that $\cM$ maps to the intersection $f^*\cL \cap \cF \subset f^*\cE$. 

\begin{defn}
Let $(\cE,\cL)$ be an exact Dirac pair.  A Courant trivialization $f^!\cE \cong \cT_Y\oplus \cT_Y^\vee$ is \defterm{compatible with $\cL$} if the subsheaf $\cF \cap \cL \subset f^*\cE$ is actually an embedded subbundle.
\end{defn}

We claim that the Lagrangian condition is equivalent to the induced map $\cM \to f^*\cL \cap \cF$ being an isomorphism, so that the data of a Lagrangian is the same as the data of a compatible Courant trivialization.  To see this, consider the intersection
\[
W = \XL \underset{\XT}{\times} \YT
\]
equipped with its 1-shifted symplectic structure, and  observe that $\YM \to \XL$ is Lagrangian if and only if $\YM \to W$ is.  Using the Lagrangian condition on $\XL \to \XT$, we may identify the tangent complex of the fibre product with the homotopy kernel of the morphism $\cT_\YT \to \coT_\XL[2]$ defined by the symplectic form.  We thus obtain an equivalence
\begin{align*}
\pi^*\cT_W &\cong \rbrac{ \xymatrix{
\cN^\vee_Y \ar[r] & \cF \oplus f^*\coT_X \ar[r] & f^*\cL^\vee \oplus \cT_Y }} \\
&\cong \rbrac{\xymatrix{\cF \oplus \coT_Y \ar[r] & f^*\cL^\vee \oplus \cT_Y}}
\end{align*}
where the component $\cF \to f^*\cL^\vee$ is induced by the nondegenerate pairing on $\cE$.  The Lagrangian condition is now equivalent to having an exact sequence
\[
(\xymatrix{\cM\ar[r] & \cT_Y}) \to \rbrac{\xymatrix{\cF \oplus \coT_Y \ar[r] & f^*\cL^\vee \oplus \cT_Y}} \to \rbrac{\xymatrix{\coT_Y \ar[r] & \cM^\vee}},
\]
which in turn is equivalent to the exactness of the complex
\[
\xymatrix{
0 \ar[r] & \cM \ar[r] & \cF \ar[r] & f^*\cL^\vee \ar[r] & \cM^\vee \ar[r] & 0
}
\]
Using the fact that $\cF$ and $\cL$ are Lagrangian subbundles, this is equivalent to having $\cM = f^*\cL \cap \cF$, as claimed.  We have therefore arrived at the following
\begin{thm}
The $\infty$-groupoid $\Lag_1(X,Y)$ consisting of one-shifted Lagrangian morphisms $\YM \to \XL$ is equivalent to the 1-groupoid exact Dirac pairs equipped with a compatible Courant trivialization along $Y$.
\end{thm}

\begin{ex}
Suppose that $\cE = \cT_X \oplus \cT_X^\vee$ is the standard Courant algebroid and $\cL = \cA \oplus \cA^\perp \subset \cE$ where $\cA\subset \cT_X$ is the involutive subbundle determined by a regular foliation of $X$ and $\cA^\perp \subset \cT^\vee_X$ is its annihilator.  If $Y \subset X$ is a union of leaves of the foliation, then the canonical trivialization $f^!\cE \cong \cT_Y\oplus \cT_Y^\vee$ is compatible with $\cL$.  Indeed, we have the generalized tangent bundle
\[
\cF = \cT_Y\oplus \cN_Y^\vee
\]
so that
\[
\cF \cap f^*\cL = f^*\cA \oplus \cN_Y^\vee
\]
which is evidently a subbundle.\qed
\end{ex}

\subsection{Isotropic and Lagrangian quotients}
\label{sec:Poisson}
Consider now the case of a  Lagrangian of the form
\[
X \to \XL,
\]
so that the Lie algebroid on the source is trivial.  Thus an isotropic structure is simply a Courant trivialization of $\cE$ on $X$, embedding  $\cL$ as a Dirac structure in the standard Courant algebroid $\cT_X\oplus \cT^\vee_X$.  According to \autoref{ex:lag-quotient}, the Lagrangian condition then forces the projection $\cL \to \coT_X$ to be an isomorphism; as is well known, such Dirac structure are precisely the graphs of Poisson bivectors.
\begin{cor}
\label{cor:poisson}
Let $X$ be a manifold.  The $\infty$-groupoid $\SympIso_1(X)$ of one-shifted isotropic quotients $X \to \XL$ is equivalent to the discrete set of Dirac structures in the standard Courant algebroid  $\cT_X\oplus \cT_X^\vee$.   Under this equivalence, the subgroupoid of Lagrangians is identified with the set of Poisson structures on $X$.
\end{cor}

The equivalence between Poisson structures and one-shifted Lagrangians is closely related to the correspondence between Poisson structures and symplectic groupoids~\cite{Karaseev1986,Weinstein1987}.  Indeed, the quotient map $X \to \XL$ gives an atlas for the stack $\XL$, via the formal groupoid
\[
G  = X \underset{\XL}{\times} X \rightrightarrows X
\]
integrating the Lie algebroid $\cL$.  Being the fibre product of Lagrangians in a one-shifted symplectic stack, $G$ carries a canonical zero-shifted symplectic structure, and this form is multiplicative by construction. 

\subsection{Lagrangian correspondences}
We close with the following interpretation of the relation between coisotropic submanifolds and Lagrangian subgroupoids~\cite{Cattaneo2004}:
\begin{prop}
For a closed submanifold $f\colon Y \to  X$, the $\infty$-groupoid of one-shifted Lagrangian correspondences
\[
\xymatrix{
Y \ar[r] \ar[d] & X \ar[d] \\
\YM \ar[r] & \XL
}
\]
is equivalent to the discrete set of Poisson structures on $X$ for which $Y$ is a coisotropic submanifold.
\end{prop}

\begin{proof}
Computing the tangent complex of the fibre product, we immediately see that the Lagrangian condition on the map 
\[
Y \to \YM \underset{\XL}{\times} X
\] is equivalent to the exactness of the sequence
\begin{align*}
\cT_Y \to \rbrac{\cM \to \cT_Y \oplus f^*\cT_X \oplus \cL \to f^*\cT_X} \to \coT_Y.
\end{align*}
The middle complex is quasi-isomorphic to $(\cM \to \cT_Y \oplus \cL)$, and using $\cL \cong \cT_X^\vee$ we get an isomorphism $\cM \cong \cN^\vee_Y$, with anchor map induced by the Poisson bivector.  Thus $Y$ is coisotropic.
\end{proof}

\bibliographystyle{hyperamsplain}
\bibliography{symplectic-algebroids.bib}

\medskip

\noindent{\small

\noindent {\sc Mathematical Institute and Jesus College, University of Oxford, UK}

\noindent \href{mailto:brent.pym@maths.ox.ac.uk}{brent.pym@maths.ox.ac.uk}

\noindent {\sc Max-Planck-Institut F\"{u}r Mathematik, Bonn, Germany}

\noindent \href{mailto:psafronov@mpim-bonn.mpg.de}{psafronov@mpim-bonn.mpg.de}}

\end{document}